\documentclass[11pt]{article}

\usepackage{amsmath,amssymb,amsfonts}
\usepackage{color,xcolor}
\usepackage{mathrsfs}
\usepackage{float}
\usepackage[linktocpage,colorlinks,linkcolor=blue,anchorcolor=blue, citecolor=blue,urlcolor=blue]{hyperref}

\newtheorem{theorem}{Theorem}[section]
\newtheorem{corollary}[theorem]{Corollary}
\newtheorem{lemma}[theorem]{Lemma}
\newtheorem{proposition}[theorem]{Proposition}

\newtheorem{definition}[theorem]{Definition}
\newtheorem{example}{Example}[section]

\newcommand{\ii}{\mbox{\bf i}}
\newcommand{\re}{\textnormal{Re}\,}
\newcommand{\im}{\textnormal{Im}\,}

\newcommand{\C}{\mathbb{C}}

\newcommand{\R}{\mathbb{R}}

\newcommand{\A}{\mathcal{A}}

\newcommand{\F}{\mathcal{F}}
\newcommand{\CL}{\mathcal{L}}
\newcommand{\G}{\mathcal{G}}
\newcommand{\HI}{\mathcal{H}}

\newcommand{\bp}{\boldsymbol{p}}
\newcommand{\bs}{\boldsymbol{s}}
\newcommand{\bx}{\boldsymbol{x}}
\newcommand{\by}{\boldsymbol{y}}
\newcommand{\bz}{\boldsymbol{z}}
\newcommand{\bw}{\boldsymbol{w}}

\newcommand{\ov}{\overline}

\newcommand{\HH}{\textnormal{H}}
\newcommand{\T}{\textnormal{T}}

\newcommand{\ex}{\textnormal{E}\,}

\begin{document}

\title{Characterizing Real-Valued Multivariate Complex Polynomials and Their Symmetric Tensor Representations}

\author{
Bo JIANG
\thanks{Research Center for Management Science and Data Analytics, School of Information Management and Engineering, Shanghai University of Finance and Economics, Shanghai 200433, China. Email: isyebojiang@gmail.com} \and
Zhening LI
\thanks{Department of Mathematics, University of Portsmouth, Portsmouth, Hampshire PO1 3HF, United Kingdom. Email: zheningli@gmail.com} \and
Shuzhong ZHANG
\thanks{Department of Industrial and Systems Engineering, University of Minnesota, Minneapolis, MN 55455, USA. Email: zhangs@umn.edu}}

\date{\today}

\maketitle

\begin{abstract}
  In this paper we study multivariate polynomial functions in complex variables and their corresponding symmetric tensor representations. The focus is to find conditions under which such complex polynomials always take real values. We introduce the notion of symmetric conjugate forms and general conjugate forms, characterize the conditions for such complex polynomials to be real-valued, and present their corresponding tensor representations. 
  New notions of eigenvalues/eigenvectors for complex tensors are introduced, extending similar properties from the Hermitian matrices. Moreover, 
  we study a property of the symmetric tensors, namely the {largest eigenvalue (in the absolute value sense)} of a symmetric real tensor is equal to its largest singular value; the result is also known as Banach's theorem. We show that a similar result holds for the complex case as well. Finally, we discuss some applications of the new notion of eigenvalues/eigenvectors for the complex tensors.

\vspace{0.5cm}


\vspace{0.5cm}


\end{abstract}

\vspace{0.5cm}
%

\section{Introduction}\label{sec:introduction}

In this paper we set out to study the functions in multivariate complex variables which however always take real values. Such functions are frequently encountered in engineering applications arising from signal processing~\cite{ADJZ12}, electrical engineering, and control theory~\cite{TO98}. It is interesting to note that such complex functions are usually not studied by conventional complex analysis, since they are typically not even analytic because the Cauchy-Riemann conditions will never be satisfied unless the function in question is trivial. There has been a surge of research attention to solve optimization models related to such kind of complex functions~\cite{ADJZ12,SBD12,SBD13,JLZ14,JMZ14}. Sorber et al.~\cite{SBD13b} developed a MATLAB toolbox for optimization problems in complex variables, where the complex function in question is either {\it pre-assumed} to be always real-valued~\cite{SBD12}, or it is the modulus/norm of a complex function~\cite{ADJZ12,SBD13}. An interesting question thus arises: {\em Can such real-valued complex functions be characterized?} Indeed there does exist a class of special complex functions that always take real values: the Hermitian quadratic form $\bx^{\HH}A\bx$ where $A$ is a Hermitian matrix. In this case, the quadratic structure plays a key role. This motivates us to search for more general complex polynomial functions with the same property. 
Interestingly, such complex polynomials can be completely characterized, as we will present in this paper.

As is well-known, polynomials can be represented by tensors. The same question can be asked about complex tensors. In fact, there is a considerable amount of recent research attention on the applications of complex tensor optimization. For instance, Hilling and Sudberythe~\cite{HS10} formulated a quantum entanglement problem as a complex multilinear form optimization under the spherical constraint, and Zhang and Qi~\cite{ZQ12} and Ni et al.~\cite{NQB14} discussed quantum eigenvalue problems, which arised from the geometric measure of entanglement of a multipartite symmetric pure state in the complex tensor space. Examples of complex polynomial optimization include Aittomaki and Koivunen~\cite{AK09} who formulated the problem of beam-pattern synthesis in array signal processing as complex quartic polynomial minimization, and Aubry et al.~\cite{ADJZ12} who modeled a radar signal processing problem by complex polynomial optimization. Solution methods for complex polynomial optimization can be found in, e.g., \cite{SBD12,JLZ14,JMZ14}. As mentioned before, polynomials and tensors are known to be related. In particular in the real domain, homogeneous polynomials (or forms) are bijectively related to {\em symmetric} tensors; i.e., the components of the tensor is invariant under the permutation of its indices. This important class of tensors generalizes the concept of symmetric matrices. As the role played by symmetric matrices in matrix theory and quadratic optimization, symmetric tensors have a profound role to play in tensor eigenvalue problems and polynomial optimization. A natural question can be asked about complex tensors: {\em What is the higher order complex tensor generalization of the Hermitian matrix?} In this paper, we manage to identify two classes of symmetric complex tensors, both of which include Hermitian matrices as a special case when the order of the tensor is two.

In recent years, the eigenvalue of tensor has become a topic of intensive research interest. {Perhaps a first attempt to generalize eigenvalue decomposition of matrices can be traced back to 2000 when De Lathauwer et al.~\cite{DDV00} introduced the so-called higher-order eigenvalue decomposition. Shortly after that, Kofidis and Regalia~\cite{KR02} showed that blind deconvolution can be formulated as a nonlinear eigenproblem.} A systematic study of eigenvalues of tensors was pioneered by Lim~\cite{L05} and Qi~\cite{Q05} independently in 2005. Various applications of tensor eigenvalues and the connections to polynomial optimization problems have been proposed;
cf.~\cite{Q07,NQWW07,ZQ12,CS13,NQB14} and the references therein. We refer the interested readers to the survey papers~\cite{Q12} for more details on the spectral theory of tensors and various applications of tensors. Computation of tensor eigenvalues is an important source for polynomial optimization~\cite{HLZ10,LHZ12}. Essentially the problem is to maximize or minimize a homogeneous polynomial under the spherical constraint, which can also be used to test the (semi)-definiteness of a symmetric tensor. 

In this paper we are primarily interested in complex polynomials/tensors that arise in the context of optimization. By nature of optimization, we are interested in the complex polynomials that always take real values. However, it is easy to see that if no {\em conjugate} term is involved, then the only class of real-valued complex polynomials is the set of real constant functions\footnote{This should be differentiated from the notion of real-symmetric complex polynomial, sometimes also called real-valued complex polynomial in abstract algebra, i.e., $\ov{f(\bx)}=f(\ov{\bx})$.}. Therefore, the conjugate terms are necessary for a complex polynomial to be real-valued. Hermitian quadratic forms mentioned earlier belong to this category, which is an active area of research in optimization; see e.g.~\cite{LMSYZ10,ZH06,SZY07}.
In the aforementioned papers~\cite{Q07,NQWW07,CS13} on eigenvalues of complex tensors, the associated complex polynomials however are not real-valued. 
The aim of this paper is different. We target for a systematic study on
the nature of symmetricity for higher order complex tensors which will lead to the property that the associated polynomials always take real values.
The main contribution of this paper is to give a full characterization for the real-valued conjugate complex polynomials and to identify two classes of symmetric complex tensors, which have already shown potentials in the algorithms design~\cite{ADJZ12,JLZ14,JMZ14}. We also proposed two new types of tensor eigenvalues/eigenvectors for the new classes of complex tensors.


This paper is organized as follows. We start with the preparation of various notations and terminologies in Section~\ref{sec:preparation}. In particular, two types of conjugate complex polynomials are defined and their symmetric tensor representations are discussed. Section~\ref{sec:condition} presents the necessary and sufficient condition for real-valued conjugate complex polynomials, based on which two types of symmetric complex tensors are defined, corresponding to the two types of real-valued conjugate complex polynomials.
As an important result in this paper, we then present the definitions and properties of eigenvalues and eigenvectors for two types of symmetric complex tensors in Section~\ref{sec:eigenvalue}. In Section~\ref{sec:Banach}, we discuss Banach's theorem, which states that the {largest eigenvalue (in the absolute value sense)} of a symmetric real tensor is equal to its largest singular value, and extend it to the two new types of symmetric complex tensors. Some application examples are discussed in Section~\ref{sec:application} to show the significance in practice of the theoretical results in this paper. Finally, we conclude this paper by summarizing our main findings and outlining possible future work in Section~\ref{sec:conclusion}.

\section{Preparation}\label{sec:preparation}

Throughout this paper we use usual lowercase letters, boldface lowercase letters, capital letters, and calligraphic letters to denote scalars, vectors, matrices, and tensors, respectively. For example, a scalar $a$, a vector $\bx$, a matrix $Q$, and a tensor $\F$. We use subscripts to denote their components, e.g. $x_i$ being the $i$-th entry of a vector $\bx$, $Q_{ij}$ being the $(i,j)$-th entry of a matrix $Q$ and $\F_{ijk}$ being the $(i,j,k)$-th entry of a third order tensor $\F$. As usual, the field of real numbers and the field of complex numbers are denoted by $\R$ and $\C$, respectively.

For any complex number $z=a+\ii b\in\C$ with $a,b\in\R$, its real part and imaginary part are denoted by $\re z:=a$ and $\im z:=b$, respectively. Its modulus is denoted by $|z|:=\sqrt{\ov{z}z}=\sqrt{a^2+b^2}$, where $\ov{z}:=a-\ii b$ denotes the conjugate of $z$. For any vector $\bx\in\C^n$, {we let $\bx^{\HH}:=\ov{\bx}^{\T}$ be the transpose of its conjugate, and we define it analogously for matrices}. Throughout this paper we uniformly use the 2-norm for vectors, matrices and tensors in general, which is the usual Euclidean norm. For example, the norm of a vector $\bx\in\C^n$ is defined as $\|\bx\|:=\sqrt{\bx^{\HH}\bx}$, and the norm of a $d$-th order tensor $\F\in\C^{n_1 \times \dots\times n_d}$ is defined as
$$
\|\F\|:= \sqrt{\sum_{i_1=1}^{n_1}\dots\sum_{i_d=1}^{n_d}\ov{\F_{i_1 \dots i_d}} \cdot \F_{i_1 \dots i_d}}.
$$

\subsection{Complex forms and their tensor representations}

A multivariate complex polynomial $f(\bx)$ is a polynomial function of variable $\bx\in\C^n$ whose coefficients are complex, e.g.\ $f(x_1,x_2)=x_1+(1-\ii){x_2}^2$. A multivariate {\em conjugate} complex polynomial (sometimes abbreviated by conjugate polynomial in this paper) $f_C(\bx)$ is a polynomial function of variables $\bx,\ov{\bx}\in\C^n$, which is differentiated by the subscript $C$, standing for `conjugate', e.g. $f_C(x_1,x_2) = x_1 + \ov{x_2} + \ov{x_1}x_2+ (1-\ii){x_2}^2$. In particular, a general $n$-dimensional $d$-th degree conjugate complex polynomial can be explicitly written as summation of monomials
$$
f_C(\bx):=\sum_{\ell=0}^d \sum_{k=0}^\ell\,\sum_{1\le i_1\le \dots \le i_k \le n}\,\sum_{1\le j_1 \le \dots \le j_{\ell-k} \le n} a_{i_1\dots i_k,j_1\dots j_{\ell-k}}\ov{x_{i_1}\dots x_{i_k}}x_{j_1} \dots x_{j_{\ell-k}}.
$$
{In the above notation for a monomial $a_{i_1\dots i_k,j_1\dots j_{\ell-k}}\ov{x_{i_1}\dots x_{i_k}}x_{j_1} \dots x_{j_{\ell-k}}$, the indices of the coefficient $a_{i_1\dots i_k,j_1\dots j_{\ell-k}}$ are always partitioned by a `,' to separate that of conjugate variables and that of regular variables. In particular, the coefficient of a monomial that only has conjugate variables such as $\ov{x_{i_1}}\ov{x_{i_2}}$ will be written as $a_{i_1i_2,}$.} In this definition, it is obvious that complex polynomials are a subclass of conjugate complex polynomials. Remark that a pure complex polynomial can never only take real values unless it is a constant. This observation follows trivially from the basic theorem of algebra.

Given a $d$-th order complex tensor $\F\in\C^{n_1 \times \dots\times n_d}$, its associated multilinear form is defined as
$$
\F(\bx^1,\dots,\bx^d):=\sum_{i_1=1}^{n_1}\dots\sum_{i_d=1}^{n_d} \F_{i_1  \dots i_d} x^1_{i_1}\dots x^d_{i_d},
$$
where $\bx^k\in \C^{n_k}$ for $k=1,\dots,d$. A complex tensor $\F\in\C^{n_1\times \dots\times n_d}$ is called {\em symmetric} if $n_1=\dots=n_d \, (=n)$ and every component $\F_{i_1 \dots i_d}$ are invariant under all permutations of the indices $\{i_1,\dots,i_d\}$. We remark that conjugation is not involved here when speaking of symmetricity for complex tensors. Closely related to a symmetric tensor $\F\in\C^{n^d}$ is a general $d$-th degree complex homogeneous polynomial function $f(\bx)$ (or complex form) of variable $\bx\in \C^n$, i.e.,
\begin{equation}\label{eq:symmetric}
f(\bx):=\F(\underbrace{\bx,\dots,\bx}_d)=\sum_{i_1=1}^{n}\dots\sum_{i_d=1}^{n} \F_{i_1 \dots i_d} x_{i_1}\dots x_{i_d}.
\end{equation}
In fact, symmetric tensors (either in the real domain or in the complex domain) are bijectively related to homogeneous polynomials; see~\cite{CGLM08}. In particular, for any $n$-dimensional $d$-th degree complex form
$$
  f(\bx)= \sum_{1\le i_1 \le \dots \le i_d\le n} a_{i_1\dots i_d}x_{i_1}\dots x_{i_d},
$$
there is a uniquely defined $n$-dimensional $d$-th order symmetric complex tensor $\F\in\C^{n^d}$ with
$$
\F_{i_1\dots i_d}=\frac{a_{i_1\dots i_d}}{|\Pi(i_1\dots i_d)|}, \quad\forall\, 1\le i_1\le \dots\le i_d\le n,
$$
satisfying~\eqref{eq:symmetric}, where $\Pi(i_1\dots i_d)$ is the set of all distinct permutations of the indices $\{i_1,\dots, i_d\}$. On the other hand, in light of formula~\eqref{eq:symmetric}, a complex form $f(\bx)$ is easily obtained from the symmetric multilinear form $\F(\bx^1,\dots,\bx^d)$ by letting $\bx^1=\dots=\bx^d=\bx$.

\subsection{Symmetric conjugate forms and their tensor representations}\label{sec:cform}

To discuss higher order conjugate complex forms and complex tensors, let us start with the {well-established} properties of Hermitian matrices. Let $A \in \C^{n^2}$ with $A^{\HH}=A$, which is not symmetric in the usual sense because $A^{\T}\neq A$ in general. The following conjugate quadratic form
$$\bx^{\HH}A\bx=\sum_{i=1}^{n}\sum_{j=1}^{n}A_{ij}\ov{x_i}x_j$$
always takes real values for any $\bx\in\C^n$. In particular, we notice that each monomial in the above form is the product of one `conjugate' variable $\ov{x_i}$ and one usual (non-conjugate) variable $x_j$.

To extend the above form to higher degrees, let us consider the following special class of conjugate polynomials:
\begin{definition} \label{def:sform}
A symmetric conjugate form of the variable $\bx\in\C^n$ is defined as
\begin{equation}\label{eq:sform}
f_S(\bx):=\sum_{1\le i_1 \le  \dots \le i_d \le n}\, \sum_{1\le j_1 \le \dots \le j_d \le n} a_{i_1\dots i_d,j_1\dots j_d}\ov{x_{i_1}\dots x_{i_d}}x_{j_1} \dots x_{j_d}.
\end{equation}
\end{definition}

Essentially, $f_S(\bx)$ is the summation of all the possible $2d$-th degree monomials that consist of $d$ conjugate variables and $d$ usual variables. Here the subscript `$S$' stands for `symmetric'. The following example is a special case of~\eqref{eq:sform}.

\begin{example}
Given a $d$-th degree complex form $h(\bx) = \sum_{1\le i_1 \le  \dots \le i_d \le n } c_{i_1\dots i_d} x_{i_1}\dots x_{i_d}$, the function
\begin{align*}
|h(\bx)|^2&=\left(\sum_{1\le i_1 \le  \dots \le i_d \le n } \ov{c_{i_1\dots i_d} x_{i_1}\dots x_{i_d}}\right) \left(\sum_{1\le j_1 \le  \dots \le j_d \le n } c_{j_1\dots j_d} x_{j_1}\dots x_{j_d}\right)\\
&= \sum_{1\le i_1 \le  \dots \le i_d \le n} \, \sum_{1\le j_1 \le  \dots \le j_d \le n } \left(\ov{c_{i_1\dots i_d}}\cdot c_{j_1\dots j_d}\right)\ov{ x_{i_1}\dots x_{i_d}}  x_{j_1}\dots x_{j_d}
\end{align*}
is a $2d$-th degree symmetric conjugate form.
\end{example}

Notice that $|h(\bx)|^2$ is actually a real-valued conjugate polynomial. Later in Section~\ref{sec:condition} we shall show that a symmetric conjugate form $f_S(\bx)$ in~\eqref{eq:sform} always takes real values if and only if the coefficients of any pair of conjugate monomials $\ov{ x_{i_1}\dots x_{i_d}} x_{j_1}\dots x_{j_d}$ and $\ov{x_{j_1}\dots x_{j_d}}x_{i_1}\dots x_{i_d}$ are conjugate to each other, i.e.,
$$
a_{i_1\dots i_d,j_1\dots j_d}=\ov{a_{j_1\dots j_d,i_1\dots i_d}}, \quad\forall \,1\le i_1\le \dots \le i_d \le n,\, 1\le j_1 \le \dots \le j_d \le n.
$$

As any complex form uniquely defines a symmetric complex tensor and vice versa, we observe a class of tensors representable for symmetric conjugate forms.
\begin{definition} \label{def:partial-symmetric}
 An even order tensor $\F\in\C^{n^{2d}}$ is called partial-symmetric if for {every} $1\le i_1\le\dots\le i_{d}\le n,\, 1\le i_{d+1}\le\dots\le i_{2d}\le n$
\begin{equation}\label{eq:partial-symmetric}
\F_{j_1\dots j_d j_{d+1}\dots j_{2d}} = \F_{i_1\dots i_d i_{d+1}\dots i_{2d}}, \quad\forall \,
(j_1\dots j_d) \in \Pi (i_1\dots i_d),\, (j_{d+1}\dots j_{2d})\in\Pi(i_{d+1}\dots i_{2d}).
\end{equation}
\end{definition}

We remark that the so-called partial-symmetricity was studied earlier in algebraic geometry by Carlini and Chipalkatti~\cite{CC03}, and was also studied in polynomial optimization~\cite{HLZ13} in the framework of mixed polynomial forms, i.e., for any fixed first $d$ indices of the tensor, it is symmetric with respect to its last $d$ indices, and vise versa. It is clear that partial-symmetricity~\eqref{eq:partial-symmetric} is weaker than the usual symmetricity for tensors.

Let us formally define the bijection $\mathbf{S}$ (taking the first initial of symmetric conjugate forms) between symmetric conjugate forms and partial-symmetric complex tensors, as follows:\\
(i) $\mathbf{S}(\F)=f_S$: Given a partial-symmetric tensor $\F\in\C^{n^{2d}}$ with its associated multilinear form $\F(\bx^1,\dots,\bx^{2d})$, the symmetric conjugate form is defined as
      $$
        f_S(\bx)=\F(\underbrace{\ov{\bx},\dots,\ov{\bx}}_d,\underbrace{\bx,\dots,\bx}_d)
        = \sum_{i_1=1}^n\dots\sum_{i_{2d}=1}^n \F_{i_1\dots i_d i_{d+1}\dots i_{2d}}\ov{x_{i_1}\dots x_{i_d}}x_{i_{d+1}} \dots x_{i_{2d}}.
      $$
(ii) $\mathbf{S}^{-1}(f_S)=\F$: Given a symmetric conjugate form $f_S$~\eqref{eq:sform}, the components of the partial-symmetric tensor $\F\in\C^{n^{2d}}$ are defined by
      \begin{equation}
        \F_{j_1\dots j_d j_{d+1}\dots j_{2d}}=\frac{a_{i_1\dots i_d ,i_{d+1}\dots i_{2d}}} {|{\Pi}(i_1\dots i_d)| \cdot |{\Pi}(i_{d+1}\dots i_{2d})|}\label{eq:sinverse}
      \end{equation}
      for all $1\le i_1\le\dots\le i_{d}\le n,\,
        1\le i_{d+1}\le\dots\le i_{2d}\le n,\,
        (j_1\dots j_d) \in \Pi (i_1\dots i_d)$ and $(j_{d+1}\dots j_{2d})\in\Pi(i_{d+1}\dots i_{2d})$.
\begin{example}
  Given a bivariate fourth degree symmetric conjugate form $f_S(\bx)=(1-\ii){\ov{x_1}}^2{x_1}^2 +4 \ov{x_1}\ov{x_2}x_1x_2+ 6 \ov{x_1}\ov{x_2}{x_2}^2 $, the corresponding partial-symmetric tensor $\F=\mathbf{S}^{-1}(f_S)\in\C^{2^4}$ satisfies that $\F_{1111}=1-\ii,\,\F_{1212}=\F_{1221}=\F_{2112}=\F_{2121}=1,\,\F_{1222}=\F_{2122}=3$ and other entries are zeros. Conversely, $f_S(\bx)$ can be obtained from $\F\big(\binom{\,\ov{x_1}\,}{\ov{x_2}}, \binom{\,\ov{x_1}\,}{\ov{x_2}}, \binom{x_1}{x_2}, \binom{x_1}{x_2}\big)$.
\end{example}

According to the mappings defined previously, the following result readily follows.
\begin{lemma}\label{thm:tensorS}
The bijection $\mathbf{S}$ is well-defined, i.e., any $n$-dimensional $2d$-th order partial-symmetric tensor $\F\in\C^{n^{2d}}$ uniquely defines an $n$-dimensional $2d$-th degree symmetric conjugate form, and vice versa.
\end{lemma}

\subsection{General conjugate forms and their tensor representations}\label{sec:gform}

In \eqref{eq:sform}, for each monomial the numbers of conjugate variables and the usual variables are always equal. This restriction can be relaxed further. 
\begin{definition} \label{def:gform}
A general conjugate form of the variable $\bx\in\C^n$ is defined as
\begin{equation}\label{eq:gform}
f_G(\bx)=\sum_{k=0}^d \, \sum_{1\le i_1\le \dots \le i_k \le n}\, \sum_{1\le j_1 \le \dots \le j_{d-k} \le n} a_{i_1\dots i_k,j_1\dots j_{d-k}}\ov{x_{i_1}\dots x_{i_k}}x_{j_1} \dots x_{j_{d-k}}.
\end{equation}
\end{definition}

Essentially, $f_G(\bx)$ is the summation of all the possible $d$-th degree monomials, allowing any number of conjugate variables as well as the usual variables in each monomial. Here the subscript `$G$' stands for `general'.
Obviously $f_S(\bx)$ is a special case of $f_G(\bx)$, and $f_G(\bx)$ is a special case of $f_C(\bx)$.

In Section~\ref{sec:condition} we shall show that a general conjugate form $f_G(\bx)$ will always take real values for all $\bx$ if and only if the coefficients of each pair of conjugate monomials are conjugate to each other.
To this end, below we shall explicitly treat the conjugate variables as new variables:\\
(i) $\mathbf{G}(\F)=f_G$: Given a symmetric tensor $\F\in\C^{(2n)^d}$ with its associated multilinear form $\F(\bx^1,\dots,\bx^d)$, the general conjugate form of $\bx\in\C^n$ is defined as
    \begin{equation}\label{eq:tensor-gform}
    f_G(\bx)=\F\bigg(\underbrace{\dbinom{\ov\bx}{\bx},\dots,\dbinom{\ov\bx}{\bx}}_d\bigg).
    \end{equation}
(ii) $\mathbf{G}^{-1}(f_G)=\F$: Given a general conjugate form $f_G$ of $\bx\in\C^n$ as~\eqref{eq:gform}, the components of the symmetric tensor $\F\in\C^{(2n)^d}$ are defined as follows: for any $1\le j_1,\dots,j_d\le 2n$, sort these $j_\ell$'s in a nondecreasing order as $1\le j_{i_1}\le\dots\le j_{i_d}\le 2n$ and let $k={\arg\max}_{1\le \ell\le d}\{j_{i_\ell}\le n\}$, then
      \begin{equation}\label{eq:tensor-gform-1}
      \F_{j_1\dots j_d} = \frac{a_{j_{i_1}\dots j_{i_k},(j_{i_{k+1}}-n)\dots (j_{i_d}-n)}}{|\Pi(j_1\dots j_d)|}.
      \end{equation}
\begin{example}
  {Given} a symmetric second order tensor (matrix) $F=\left( \begin{smallmatrix} \ii & 0 & 1 & 0\\ 0 & 0 & 2 & 0\\ 1 & 2 & 0 & 0 \\ 0 & 0 &0 & 3 \end{smallmatrix}  \right)\in\C^{4^2}$,  the corresponding general conjugate form is $$f_G(\bx)=(\ov{x_1},\ov{x_2},x_1,x_2)F(\ov{x_1},\ov{x_2},x_1,x_2)^{\T}=\ii{\ov{x_1}}^2 + 2\ov{x_1}x_1 + 4 \ov{x_2}x_1 + 3 {x_2}^2.$$
  Conversely, $F=\mathbf{G}^{-1}(f_G)$ can obtained component-wisely by~\eqref{eq:tensor-gform-1}.
\end{example}

Similar {to} Lemma~\ref{thm:tensorS}, the following is easily verified; we leave its proof to the interested readers.
\begin{lemma}\label{thm:tensorG}
The bijection $\mathbf{G}$ is well-defined, i.e., any $2n$-dimensional $d$-th order symmetric tensor $\F\in\C^{(2n)^d}$ uniquely defines an $n$-dimensional $d$-th degree general conjugate form, and vice versa.
\end{lemma}

To conclude this section we remark that a partial-symmetric tensor (representation for a symmetric conjugate form) is less restrictive than a symmetric tensor (representation for a general conjugate form), while a symmetric conjugate form is a special case of a general conjugate form. One should note that the dimensions of these two tensor representations are actually different.

\section{Real-valued conjugate forms and their tensor representations}\label{sec:condition}

In this section, we study the two types of conjugate complex forms introduced in Section~\ref{sec:preparation}: symmetric conjugate forms and general conjugate forms.

\subsection{Real-valued conjugate polynomials}
Let us first focus on polynomials, and present {the following general characterization of real-valued conjugate complex polynomials}.

\begin{theorem}\label{thm:realvalue}
  A conjugate complex polynomial function is real-valued if and only if the coefficients of any pair of its conjugate monomials are conjugate to each other, i.e., any two monomials $au_C(\bx)$ and $bv_C(\bx)$ with $a$ and $b$ being their coefficients satisfying $\ov{u_C(\bx)}=v_C(\bx)$ must have that $\ov{a}=b$.
\end{theorem}

The above condition actually implies that the coefficient of any self-conjugate monomial must be real. Applying Theorem~\ref{thm:realvalue} to the two classes of conjugate forms that we just introduced, the conditions for them to always take real values can now be characterized:
\begin{corollary}\label{thm:condition}
A symmetric conjugate form
$$
f_S(\bx)=\sum_{1\le i_1 \le  \dots \le i_d \le n}\, \sum_{1\le j_1 \le \dots \le j_d \le n} a_{i_1\dots i_d,j_1\dots j_d}\ov{x_{i_1}\dots x_{i_d}}x_{j_1} \dots x_{j_d}
$$
is real-valued if and only if
\begin{equation}\label{eq:sform-condition}
a_{i_1\dots i_d,j_1\dots j_d}=\ov{a_{j_1\dots j_d,i_1\dots i_d}}, \quad\forall\,1\le i_1\le \dots \le i_d \le n, \,1\le j_1 \le \dots \le j_d \le n.
\end{equation}
A general conjugate form
$$
f_G(\bx)=\sum_{k=0}^d \, \sum_{1\le i_1\le \dots \le i_k \le n}\, \sum_{1\le j_1 \le \dots \le j_{d-k} \le n} a_{i_1\dots i_k,j_1\dots j_{d-k}}\ov{x_{i_1}\dots x_{i_k}}x_{j_1} \dots x_{j_{d-k}}
$$
is real-valued if and only if
$$
a_{i_1\dots i_k,j_1\dots j_{d-k}}= \ov{a_{j_1\dots j_{d-k},i_1\dots i_k}}, \quad\forall\,1\le i_1 \le \dots \le i_k \le n, \,1\le j_1 \le \dots \le j_{d-k} \le n,\, 0\leq k \leq d.
$$
\end{corollary}

Let us now prove Theorem~\ref{thm:realvalue}. We first show the `if' part of the theorem, which is quite straightforward. To see this, for any pair of conjugate monomials (including self-conjugate monomial as a special case) of a conjugate complex polynomial: $au_C(\bx)$ and $b\ov{u_C(\bx)}$ with $a,b\in\C$ being their coefficients, 
if $\ov{a}=b$, then
$$
\ov{au_C(\bx)+b\ov{u_C(\bx)}} = \ov{au_C(\bx)+\ov{a}\ov{u_C(\bx)}} = au_C(\bx) + \ov{a}\ov{u_C(\bx)} = au_C(\bx)+b\ov{u_C(\bx)},
$$
implying that $a u_C(\bx)+b \ov{u_C(\bx)}$ is real-valued. Since all the conjugate monomials of a conjugate complex polynomial can be partitioned by conjugate pairs and self-conjugate monomials, the result follows immediately.


To proceed to the `only if' part of the theorem, let us first consider an easier case of univariate conjugate polynomials.
\begin{lemma}\label{thm:uni-zero}
A univariate conjugate complex polynomial $\sum_{\ell=0}^d\sum_{k=0}^{\ell} b_{k,\ell-k}\ov{x}^kx^{\ell-k}=0$ for all $x\in\C$ if and only if all its coefficients are zeros, i.e., $b_{k,\ell-k}=0$ for all $0\le \ell\le d$ and $0\le k\le \ell$.
\end{lemma}
\begin{proof}
Let $x = \rho e^{\ii \theta}$ with $\rho\ge0$ and $\theta\in[0,2\pi)$, and the identity can be rewritten as
\begin{equation}\label{eq:unipoly}
\sum_{\ell=0}^d\left(\sum_{k=0}^{\ell} b_{k,\ell-k} e^{\ii(\ell-2k)\theta}\right)\rho^\ell  = 0.
\end{equation}
For any fixed $\theta$, the function can be viewed as a polynomial with respect to $\rho$. Therefore the coefficient of the highest degree monomial $\rho^d$ must be zero, i.e.,
$$
\sum_{k=0}^d b_{k,d-k} e^{\ii (d-2k)\theta} = 0, \quad \forall\, \theta\in[0,2\pi).
$$
Consequently we have for any $\theta\in[0,2\pi)$,
\begin{align}
&\sum_{k=0}^d \re(b_{k,d-k})\cos ((d-2k)\theta)-\sum_{k=0}^d \im(b_{k,d-k})\sin ((d-2k)\theta) =0,
\label{eq:unipoly1} \\
&\sum_{k=0}^d \im(b_{k,d-k})\cos ((d-2k)\theta) + \sum_{k=0}^d \re( b_{k,d-k}) \sin ((d-2k)\theta)=0. \label{eq:unipoly2}
\end{align}
The first and second parts of~\eqref{eq:unipoly1} can be respectively simplified as
\begin{align*}
&~~~~\sum_{k=0}^d \re(b_{k,d-k}) \cos ((d-2k)\theta)\\
&=\left\{
\begin{array}{ll}
\sum_{k=0}^{\frac{d-1}{2}}\re(b_{k,d-k}+b_{d-k,k})\cos((d-2k)\theta) & d\mbox{ is odd}\\
\sum_{k=0}^{ \frac{d-2}{2}}\re(b_{k,d-k}+b_{d-k,k}) \cos((d-2k)\theta) + \re(b_{d/2,d/2}) & d\mbox{ is even}
 \end{array}
 \right.
\end{align*}
and
$$
\sum_{k=0}^d \im(b_{k,d-k}) \sin ((d-2k)\theta)= \sum_{k=0}^{\lfloor \frac{d-1}{2}\rfloor}
\im(b_{k,d-k}-b_{d-k,k}) \sin((d-2k)\theta).
$$
By the orthogonality of the trigonometric functions, the above further leads to
$$
\re(b_{k,d-k}+b_{d-k,k})=\im(b_{k,d-k}-b_{d-k,k})=0, \quad\forall\,k=0,1,\dots,d.
$$
Similarly,~\eqref{eq:unipoly2} implies
$$
\re(b_{k,d-k}-b_{d-k,k})=\im(b_{k,d-k}+b_{d-k,k})=0, \quad\forall\,k=0,1,\dots,d.
$$
Combining the above two sets of identities yields
$$
b_{k,d-k}=0, \quad\forall\, k=0,1,\dots,d.
$$
The degree of the function in~\eqref{eq:unipoly} (in terms of $\rho$) is then reduced by 1. The desired result follows obviously.
\end{proof}

Let us now extend Lemma~\ref{thm:uni-zero} to general multivariate conjugate polynomials.
\begin{lemma}\label{thm:multi-zero}
An $n$-dimensional $d$-th degree conjugate complex polynomial 
$$
  f_C(\bx)=\sum_{\ell=0}^d \sum_{k=0}^\ell\, \sum_{1\le i_1\le \dots \le i_k \le n}\, \sum_{1\le j_1 \le \dots \le j_{\ell-k} \le n} b_{i_1\dots i_k,j_1\dots j_{\ell-k}}\ov{x_{i_1}\dots x_{i_k}}x_{j_1} \dots x_{j_{\ell-k}} = 0
$$
for all $\bx\in\C^n$ if and only if all its coefficients are zeros, i.e.,
$b_{i_1\dots i_k,j_1\dots j_{\ell-k}}=0$ for all $0\le\ell\le d$, $0\le k\le \ell$,
$1\le i_1 \le\dots \le i_k \le n$ and $1\le j_1 \le\dots \le j_{d-k} \le n$.
\end{lemma}
\begin{proof}
We shall prove the result by induction on the dimension $n$. The case $n=1$ is already shown in Lemma~\ref{thm:uni-zero}. Suppose the claim holds for all positive integers no more than $n-1$. Then for the dimension $n$, the conjugate polynomial $f_C(\bx)$ can be rewritten according to the degrees of $\ov{x_1}$ and $x_1$ as
$$
f_C(\bx) = \sum_{\ell=0}^d\sum_{k=0}^{\ell}\ov{x_1}^k {x_1}^{\ell-k} h_C^{\ell k} (x_2,\dots,x_n).
$$
For any given $x_2,\dots,x_n\in\C$, taking $f_C$ as a univariate conjugate polynomial of $x_1$, by Lemma~\ref{thm:uni-zero} we have
$$
h_C^{\ell k}(x_2,\dots,x_n)=0, \quad\forall\, 0\le \ell\le d,\, 0\le k\le \ell.
$$
For any given $(\ell,k)$, as $h_C^{\ell k}(x_2,\dots,x_n)$ is a conjugate polynomial of dimension at most $n-1$, by the induction hypothesis all the coefficients of $h_C^{\ell k}$ are zeros. Observing that all the coefficients of $f_C$ are distributed in the coefficients of $h_C^{\ell k}$ for all $(\ell,k)$, the result is proven for dimension $n$.
\end{proof}

With Lemma~\ref{thm:multi-zero} at hand, we can finally complete the `only if' part of Theorem~\ref{thm:realvalue}. Suppose a conjugate polynomial $f(\bx)$ is real-valued for all $\bx\in\C^n$. Clearly we have $f(\bx)-\ov{f(\bx)}=0$ for all $\bx\in\C^n$, i.e.,
$$
 \sum_{\ell=0}^d \sum_{k=0}^\ell \,
 \sum_{1\le i_1\le \dots \le i_k \le n} \,\sum_{1\le j_1 \le \dots \le j_{\ell-k} \le n}
 \left(b_{i_1\dots i_k,j_1\dots j_{\ell-k}}-\ov{b_{j_1\dots j_{\ell-k},i_1\dots i_k}}\right)
 \ov{x_{i_1}\dots x_{i_k}}x_{j_1} \dots x_{j_{\ell-k}}=0.
$$
By Lemma~\ref{thm:multi-zero} it follows that $b_{i_1\dots i_k,j_1\dots j_{\ell-k}} -\ov{b_{j_1\dots j_{\ell-k},i_1\dots i_k}}=0$ for all $0\le\ell\le d$, $0\le k\le \ell$, $1\le i_1 \le\dots \le i_k \le n$ and $1\le j_1 \le\dots \le j_{d-k} \le n$, proving the `only if' part of Theorem~\ref{thm:realvalue}.

With Theorem~\ref{thm:realvalue}, in particular Corollary~\ref{thm:condition}, we are in a position to characterize the tensor representations for real-valued conjugate forms. Before concluding this subsection, let us present an alternative representation of real-valued symmetric conjugate forms, as a consequence of Corollary~\ref{thm:condition}.

\begin{proposition} \label{thm:polyrealvalued}
A symmetric conjugate form $f_S(\bx)$ is real-valued if and only if $$f_S(\bx)=\sum_{k=1}^m\alpha_k |h_k(\bx)|^2,$$ where $h_k(\bx)$ is a complex form and $\alpha_k\in\R$ for all $1\le k\le m$.
\end{proposition}
\begin{proof}
The `if' part is trivial. Next we prove the `only if' part of the proposition. If $f_S(\bx)$ is real-valued, by Corollary~\ref{thm:condition} we have~\eqref{eq:sform-condition}. Then for any $1\le i_1\le \dots \le i_d \le n$ and $1\le j_1 \le \dots \le j_d \le n$, the sum of the conjugate pair satisfies
\begin{align*}
&~~~~a_{i_1\dots i_d,j_1\dots j_d}\ov{ x_{i_1}\dots x_{i_d}}x_{j_1} \dots x_{j_d}+ a_{j_1\dots j_d,i_1\dots i_d}\ov{ x_{j_1} \dots x_{j_d}}x_{i_1}\dots x_{i_d}\\
&=a_{i_1\dots i_d,j_1\dots j_d}\ov{ x_{i_1}\dots x_{i_d}}x_{j_1} \dots x_{j_d}+ \ov{a_{i_1\dots i_d,j_1\dots j_d}}\ov{ x_{j_1} \dots x_{j_d}}x_{i_1}\dots x_{i_d}\\
&=|x_{i_1}\dots x_{i_d} + a_{i_1\dots i_d,j_1\dots j_d}x_{j_1}
\dots x_{j_d} |^2-|x_{i_1}\dots x_{i_d}|^2-|a_{i_1\dots
i_dj_1\dots j_d}x_{j_1} \dots x_{j_d} |^2.
\end{align*}
Summing up all such pairs (taking half if it is a self-conjugate pair), the conclusion follows.
\end{proof}

Similarly we have the following result for general conjugate forms.
\begin{proposition} A general conjugate form $f_G(\bx)$ is real-valued if and only if
$$f_G(\bx)=\sum_{k=1}^m \alpha_k |h_k(\bx)|^2,$$
where $h_k(\bx)$ is a complex polynomial and $\alpha_k\in\R$ for all $1\le k\le m$.
\end{proposition}

\subsection{Conjugate partial-symmetric tensors} \label{sec:ctensor}

As any symmetric conjugate form uniquely defines a partial-symmetric tensor (Lemma~\ref{thm:tensorS}), it is interesting to see more structured tensor representations for real-valued symmetric conjugate forms.

\begin{definition}\label{thm:cps}
    An even order tensor $\F\in\C^{n^{2d}}$ is called conjugate partial-symmetric if\\
    (i) $\F_{i_1\dots i_d i_{d+1}\dots i_{2d}} = \F_{j_1\dots j_d j_{d+1}\dots j_{2d}}$ for all $(j_1\dots j_d) \in \Pi (i_1\dots i_d)$ and $(j_{d+1}\dots j_{2d})\in\Pi(i_{d+1}\dots i_{2d})$, and \\
    (ii) $\F_{i_1\dots i_d i_{d+1}\dots i_{2d}} = \ov{\F_{i_{d+1}\dots i_{2d} i_1\dots i_d}}$\\
    hold for all $1\le i_1\le\dots\le i_{d}\le n$ and $1\le i_{d+1}\le\dots\le i_{2d}\le n$.
\end{definition}

We remark that when $d=1$, a conjugate partial-symmetric tensor is simply a Hermitian matrix. For a general even degree, the square matrix flattening of a conjugate partial-symmetric tensor; i.e., flattening a tensor in $\C^{n^{2d}}$ to a matrix in $\C^{(n^d)^2}$ by grouping the tensor's first $d$ modes into the rows of the matrix and its last $d$ modes into the columns of the matrix, is actually a Hermitian matrix.
The conjugate partial-symmetric tensors and the real-valued symmetric conjugate forms are connected as follows.

\begin{proposition} \label{thm:conju-mapping}
Any $n$-dimensional $2d$-th order conjugate partial-symmetric tensor $\F\in\C^{n^{2d}}$ uniquely defines (under $\mathbf{S}$) an $n$-dimensional $2d$-th degree real-valued symmetric conjugate form, and vice versa (under $\mathbf{S}^{-1}$).
\end{proposition}
\begin{proof}
For any conjugate partial-symmetric tensor $\F$, $f_S=\mathbf{S}(\F)$ satisfies
\begin{align*}
\ov{f_S(\bx)}=\ov{\F(\underbrace{\ov{\bx},\dots,\ov{\bx}}_d,\underbrace{\bx,\dots,\bx}_d)}
&= \sum_{i_1=1}^n\dots\sum_{i_{2d}=1}^n \ov{\F_{i_1\dots i_d i_{d+1}\dots i_{2d}}\ov{x_{i_1}\dots x_{i_d}}x_{i_{d+1}} \dots x_{i_{2d}}}\\
&= \sum_{i_1=1}^n\dots\sum_{i_{2d}=1}^n \ov{\F_{i_1\dots i_d i_{d+1}\dots i_{2d}}}x_{i_1}\dots x_{i_d}\ov{x_{i_{d+1}} \dots x_{i_{2d}}}\\
&= \sum_{i_1=1}^n\dots\sum_{i_{2d}=1}^n \F_{i_{d+1}\dots i_{2d}i_1\dots i_d}\ov{x_{i_{d+1}} \dots x_{i_{2d}}}x_{i_1}\dots x_{i_d}\\
&= f_S(\bx),
\end{align*}
implying that $f_S$ is real-valued.

On the other hand, for any real-valued symmetric conjugate form $f_S(\bx)$ in~\eqref{eq:sform}, it follows from Corollary~\ref{thm:condition} that $a_{i_1\dots i_d,j_1\dots j_d}=\ov{a_{j_1\dots j_d,i_1\dots i_d}}$ holds for all the possible $(i_1,\dots,i_d,j_1,\dots,j_d)$. By~\eqref{eq:sinverse}, its tensor representation $\F=\mathbf{S}^{-1}(f_S)$ with
$$
\F_{i_1\dots i_d i_{d+1}\dots i_{2d}}=\frac{a_{i_1\dots i_d ,i_{d+1}\dots i_{2d}}} {|{\Pi}(i_1\dots i_d)| \cdot |{\Pi}(i_{d+1}\dots i_{2d})| }
$$
satisfies the 2nd condition in Definition~\ref{thm:cps}, proving the conjugate partial-symmetricity of $\F$.
\end{proof}

Below is a useful property for conjugate partial-symmetric tensors, in the same vein as Proposition~\ref{thm:polyrealvalued} for the real-valued symmetric conjugate forms.
\begin{proposition} \label{thm:conju-decomp}
  An even order tensor $\F\in\C^{n^{2d}}$ is conjugate partial-symmetric if and only if
  $$
  \F = \sum_{k=1}^m\alpha_k\ov{\HI_k}\otimes \HI_k,
  $$
  where $\HI_k\in\C^{n^d}$ is symmetric and $\alpha_k\in\R$ for all $1\le k\le m$.
\end{proposition}
\begin{proof}
According to Definition~\ref{thm:cps}, it is straightforward to verify that $\sum_{k=1}^m\alpha_k\ov{\HI_k}\otimes \HI_k$ is conjugate partial-symmetric, proving the `if' part of the proposition. Let us now prove the `only if' part.

By Proposition~\ref{thm:conju-mapping}, $\mathbf{S}(\F)$ is a real-valued symmetric conjugate form. Further by Proposition~\ref{thm:polyrealvalued}, $\mathbf{S}(\F)$ can be written as
$$\F(\underbrace{\ov{\bx},\dots,\ov{\bx}}_d, \underbrace{\bx,\dots,\bx}_d)=\sum_{k=1}^m\alpha_k |h_k(\bx)|^2,$$
where $h_k(\bx)$ is a complex form and $\alpha_k\in\R$ for $k=1,\dots,m$. Let $\HI_k\in\C^{n^d}$ be the symmetric complex tensor associated with the complex form $h_k(\bx)$ for $k=1,\dots,m$; i.e.,
$$\HI_k(\underbrace{\bx,\dots,\bx}_d)=h_k(\bx).$$
We have
\begin{equation}\label{eq:HHexpand}
|h_k(\bx)|^2
=\ov{h_k(\bx)} h_k(\bx)
=\ov{\HI_k}(\underbrace{\ov{\bx},\dots,\ov{\bx}}_d) \HI_k(\underbrace{\bx,\dots,\bx}_d)
=(\ov{\HI_k}\otimes \HI_k)(\underbrace{\ov{\bx},\dots,\ov{\bx}}_d,\underbrace{\bx,\dots,\bx}_d).
\end{equation}
Thus, the symmetric conjugate form $\mathbf{S}\left(\sum_{k=1}^m\alpha_k\ov{\HI_k}\otimes \HI_k\right)$ satisfies
\begin{align*}
  \left(\sum_{k=1}^m\alpha_k\ov{\HI_k}\otimes \HI_k\right) (\underbrace{\ov{\bx},\dots,\ov{\bx}}_d,\underbrace{\bx,\dots,\bx}_d)
  &= \sum_{k=1}^m\alpha_k(\ov{\HI_k}\otimes \HI_k)(\underbrace{\ov{\bx},\dots,\ov{\bx}}_d,\underbrace{\bx,\dots,\bx}_d) \\
  &= \sum_{k=1}^m\alpha_k |h_k(\bx)|^2 \\
  &= \F(\underbrace{\ov{\bx},\dots,\ov{\bx}}_d, \underbrace{\bx,\dots,\bx}_d),
\end{align*}
i.e., $\mathbf{S}\left(\sum_{k=1}^m\alpha_k\ov{\HI_k}\otimes \HI_k\right)=\mathbf{S}(\F)$. As $\mathbf{S}$ is bijective, we have $\F=\sum_{k=1}^m\alpha_k\ov{\HI_k}\otimes \HI_k$.
\end{proof}

\subsection{Conjugate super-symmetric tensors} \label{sec:gtensor}

Similar as for real-valued symmetric conjugate forms, we have the following tensor representations for real-valued general conjugate forms.
\begin{definition}\label{thm:css}
An even dimensional tensor $\F\in\C^{(2n)^d}$ is called conjugate super-symmetric if \\
{(i) $\F$ is symmetric, 
and \\
(ii) $\F_{i_1\dots i_d} = \ov{\F_{j_1\dots j_d }}$ holds for all $1\le i_1, \dots, i_d, j_1,\dots, j_d\le 2n$ with $|i_k -j_k| = n$ for $k=1,\dots,d$.}
\end{definition}

We remark that the conjugate super-symmetricity is actually {\em stronger} than the ordinary symmetricity for complex tensors since a second condition in Definition~\ref{thm:css} is required. Actually, this condition is to ensure that the general conjugate form
$$\mathbf{G}(\F)(\bx)=\F\bigg(\underbrace{\binom{\ov\bx}{\bx},\dots,\binom{\ov\bx}{\bx}}_d\bigg)$$
is real-valued. This is because if $|i_k -j_k| = n$ holds for $k=1,\dots,d$, then the monomial with coefficient $\F_{i_1\dots i_d}$ and the monomial with coefficient $\F_{j_1\dots j_d}$ in the above form are actually a conjugate pair by noticing that the position of a conjugate variable $\ov{x_i}$ and that of a usual variable $x_i$ in the vector $\binom{\ov{\bx}}{\bx}$ differs exactly {by} $n$ for every $i$. Under the mapping $\mathbf{G}$ defined in Section~\ref{sec:gform}, it is straightforward to verify the following tensor representations for real-valued general conjugate forms. 

\begin{proposition}\label{thm:gform-rv}
Any $2n$-dimensional $d$-th order conjugate super-symmetric tensor $\F\in\C^{(2n)^d}$ uniquely defines (under $\mathbf{G}$) an $n$-dimensional $d$-th degree real-valued general conjugate form, and vice versa (under $\mathbf{G}^{-1}$).
\end{proposition}

\section{Eigenvalues and eigenvectors of complex tensors}\label{sec:eigenvalue}

As mentioned earlier, Lim~\cite{L05} and Qi~\cite{Q05} independently proposed to systematically study the eigenvalues and eigenvectors for real tensors. 
Subsequently, the topic has attracted much attention due to the potential applications in magnetic resonance imaging, polynomial optimization theory, quantum physics, statistical data analysis, higher order Markov chains, and so on. After that, this study was also extended to complex tensors~\cite{Q07,NQWW07,CS13}
without considering the conjugate variables.
Zhang and Qi in~\cite{ZQ12} proposed the so-called Q-eigenvalues of complex tensors.
\begin{definition}[Zhang and Qi~\cite{ZQ12}]\label{thm:Qeigen}
A scalar $\lambda$ is called a Q-eigenvalue of a symmetric complex tensor $\HI$, if there exists a vector $\bx$ called Q-eigenvector, such that
\begin{equation}\label{eq:Qeigen}
\left\{
\begin{array}{l}
\HI(\bullet,\underbrace{\bx,\dots,\bx}_{d-1})=\lambda\ov\bx\\
\bx^{\HH}\bx=1\\
\lambda \in \R.
\end{array}
\right.
\end{equation}
\end{definition}

Throughout this paper, the notation `$\bullet$' stands for a position left for a vector entry. In Definition~\ref{eq:Qeigen}, as the corresponding complex tensor does not have conjugate-type symmetricity, the Q-eigenvalue does not specialize to the classical eigenvalues of Hermitian matrices. In particular, $\lambda\in\R$ is required in the system~\eqref{eq:Qeigen}. Later on, Ni et al.~\cite{NQB14} defined the notion of unitary symmetric eigenvalue (US-eigenvalue) {and demonstrated a relation with the geometric measure of quantum entanglement}.
\begin{definition}[Ni et al.~\cite{NQB14}]\label{thm:USeigen}
A scalar $\lambda$ is called a US-eigenvalue of a symmetric complex tensor $\HI$, if there exists a vector $\bx$ called US-eigenvector, such that
\begin{equation}\label{eq:USeigen}
\left\{
\begin{array}{l}
\ov\HI(\bullet,\underbrace{\bx,\dots,\bx}_{d-1})=\lambda\ov\bx\\
\HI(\bullet,\underbrace{\ov\bx,\dots,\ov\bx}_{d-1})=\lambda\bx\\
\bx^{\HH}\bx=1.
\end{array}
\right.
\end{equation}
\end{definition}

{
In fact, the Q-eigenvalue and the US-eigenvalue are essentially the same.
\begin{proposition}
  $(\lambda, \bx)$ is a pair of Q-eigenvalue and Q-eigenvector if and only if $(\lambda, \ov{\bx})$ is a pair of US-eigenvalue and US-eigenvector.
\end{proposition}
\begin{proof}
  First, Definition~\ref{thm:USeigen} implies that a US-eigenvalue is always real. To see this, pre-multiplying $\bx^{\T}$ to the first equation of~\eqref{eq:USeigen} gives
  $$\ov\HI(\underbrace{\bx,\dots,\bx}_{d})=\lambda\bx^{\T}\ov\bx=\lambda,$$
  and pre-multiplying $\ov\bx^{\T}$ to the second equation of~\eqref{eq:USeigen} yields
  $$\HI(\underbrace{\ov\bx,\dots,\ov\bx}_{d})=\lambda\ov\bx^{\T}\bx=\lambda.$$
  Therefore $\ov\lambda=\lambda$ and so $\lambda\in\R$. This actually implies that the first and second equations of~\eqref{eq:USeigen} are the same by applying the conjugation to the second one. Thus, ~\eqref{eq:USeigen} is equivalent to
    \begin{equation*}
    \left\{
    \begin{array}{l}
    \HI(\bullet,\underbrace{\ov\bx,\dots,\ov\bx}_{d-1})=\lambda\bx\\
    \bx^{\HH}\bx=1\\
    \lambda\in\R.
    \end{array}
    \right.
    \end{equation*}
  The claimed equivalence is obvious by comparing the above system with~\eqref{eq:Qeigen}.
\end{proof}

In terms of eigenvalues, Definitions~\ref{thm:Qeigen} and~\ref{thm:USeigen} are the same.} Now with all the new notions introduced in the previous sections---in particular the bijection between conjugate partial-symmetric tensors and real-valued symmetric conjugate forms, and the bijection between conjugate super-symmetric tensors and real-valued general conjugate forms---we are able to present new definitions and properties of eigenvalues for complex tensors, which are naturally related to that of Hermitian matrices.

\subsection{Definitions and properties of eigenvalues} \label{sec:eigendef}

Let us first introduce two types of eigenvalues for conjugate partial-symmetric tensors and conjugate super-symmetric tensors.
\begin{definition}\label{thm:Ceigen}
$\lambda\in\C$ is called a C-eigenvalue of a conjugate partial-symmetric tensor $\F$, if there exists a vector $\bx \in \C^n$ called C-eigenvector, such that
\begin{equation}\label{eq:Ceigen}
\left\{
\begin{array}{l} \F(\bullet,\underbrace{\ov{\bx},\dots,\ov{\bx}}_{d-1}, \underbrace{\bx,\dots,\bx}_d)=\lambda \bx\\
\bx^{\HH}\bx=1.
 \end{array}\right.
\end{equation}
\end{definition}

\begin{definition}\label{thm:Geigen}
$\lambda\in\C$ is called a G-eigenvalue of a conjugate super-symmetric tensor $\F$, if there exists a vector $\bx \in \C^n$ called G-eigenvector, such that
\begin{equation}\label{eq:Geigen}
\left\{
\begin{array}{l} \F\bigg(\dbinom{\bullet}{\bullet}, \underbrace{\dbinom{\ov\bx}{\bx}, \dots,
\dbinom{\ov\bx}{\bx}}_{d-1} \bigg)=\lambda \dbinom{\bx}{\ov\bx}\\
\bx^{\HH}\bx=1.
 \end{array}\right.
\end{equation}
\end{definition}

In fact, the two types of eigenvalues defined above are always real, although they are defined in the complex domain. This property generalizes the well-known property of Hermitian matrices. In particular, Definition~\ref{thm:Ceigen} includes eigenvalues of Hermitian matrices as a special case when $d=1$.
\begin{proposition}\label{thm:eigen-real}
{Every} $C$-eigenvalue of a conjugate partial-symmetric tensor is always real; so is {every} $G$-eigenvalue of a conjugate super-symmetric tensor.
\end{proposition}
\begin{proof} Suppose $(\lambda,\bx)$ is a C-eigenvalue and C-eigenvector pair of a conjugate partial-symmetric tensor $\F$. Multiplying ${\ov\bx}^{\T}$ on both sides of the first equation in~\eqref{eq:Ceigen}, we get
$$
\F(\underbrace{\ov{\bx},\dots,\ov{\bx}}_d,\underbrace{\bx,\dots,\bx}_d) =\lambda \ov\bx^{\T}\bx=\lambda.
$$
As $\F$ is conjugate partial-symmetric, the left hand side of the above equation is real-valued, and so is $\lambda$.

Next, suppose $(\lambda,\bx)$ is a G-eigenvalue and G-eigenvector pair of a conjugate super-symmetric tensor $\F$. Multiplying $\binom{\ov\bx}{\bx}^{\T}$ on both sides of the first equation in~\eqref{eq:Geigen} yields
$$
\F\bigg(\underbrace{\dbinom{\ov\bx}{\bx}, \dots, \dbinom{\ov\bx}{\bx}}_d \bigg)=\lambda {\binom{\ov\bx}{\bx}}^{\T} \dbinom{\bx}{\ov\bx} = 2 \lambda\bx^{\HH}\bx = 2\lambda.
$$
As $\F$ is conjugate super-symmetric, the left hand side of the above equation is real-valued, and so is $\lambda$.
\end{proof}

As a consequence of Proposition~\ref{thm:eigen-real},
one can similarly define the C-eigenvalue $\lambda\in\R$ and its corresponding C-eigenvector $\bx\in\C^n$ for a conjugate partial-symmetric tensor $\F$ equivalently as follows.
\begin{proposition}\label{thm:Ceigen2}
$\lambda\in\C$ is a C-eigenvalue of a conjugate partial-symmetric tensor $\F$, if and only if there exists a vector $\bx\in\C^n$, such that
\begin{equation}\label{eq:Ceigen2}
\left\{
\begin{array}{l} \F(\underbrace{\ov{\bx},\dots,\ov{\bx}}_{d}, \underbrace{\bx,\dots,\bx}_{d-1},\bullet) =\lambda \ov\bx\\
\bx^{\HH}\bx=1.
 \end{array}\right.
\end{equation}
\end{proposition}

One important property of the Z-eigenvalues for real symmetric tensors is that they can be fully characterized by the KKT solutions of a certain optimization problem~\cite{L05,Q05}. At a first glance, this property may not hold for C-eigenvalues and G-eigenvalues since the real-valued complex functions are not analytic. Therefore, direct extension of the KKT condition of an optimization problem with such objective function may not be valid. However, this class of functions is indeed analytic if we treat the complex variables and their conjugates as a whole due to the so-called Wirtinger calculus~\cite{R91} developed in the early 20th century. In the optimization context, without noticing the Wirtinger calculus, Brandwood~\cite{B83} first proposed the notion of complex gradient. In particular, the gradient of a real-valued complex function can be taken as $\left(\frac{\partial}{\partial \bx}, \frac{\partial}{\partial \ov\bx}\right)$. Interested readers are referred to~\cite{SBD12} for more discussions on the Wirtinger calculus in optimization with complex variables.

With the help of Wirtinger calculus, we are able to characterize C-eigenvalues and C-eigenvectors in terms of the KKT solutions. Therefore many optimization techniques can be applied to find the C-eigenvalues/eigenvectors for a conjugate partial-symmetric tensor.
\begin{proposition}\label{prop:C-eigen}
$\bx\in\C^n$ is a C-eigenvector associated with a C-eigenvalue $\lambda\in\R$ for a conjugate partial-symmetric tensor $\F$ if and only if $\bx$ is a KKT point of the optimization problem
$$
\max_{\bx^{\HH}\bx=1}\F(\underbrace{\ov\bx,\dots,\ov\bx}_d, \underbrace{\bx,\dots,\bx}_d)
$$
with Lagrange multiplier being $d\lambda$ and the corresponding objective value being $\lambda$.
\end{proposition}
\begin{proof}
{By the multilinearity of $\F$, the gradient on $\bx$ of the real-valued symmetric conjugate form associated with $\F$ is given by
$$
\F(\underbrace{\ov{\bx},\dots,\ov{\bx}}_{d},\bullet, \underbrace{\bx,\dots,\bx}_{d-1}) + \cdots + \F(\underbrace{\ov{\bx},\dots,\ov{\bx}}_{d},\underbrace{\bx,\dots,\bx}_{d-1},\bullet) = d \cdot \F(\underbrace{\ov{\bx},\dots,\ov{\bx}}_{d},\bullet, \underbrace{\bx,\dots,\bx}_{d-1}),
$$
where the equality is due to the partial-symmetry of $\F$.}

Denote $\mu$ to be the Lagrange multiplier associated with the constraint $\bx^{\HH}\bx=1$. 
The KKT condition gives rise to the equations
$$
\left\{
\begin{array}{l} d \cdot \F(\bullet, \underbrace{\ov{\bx},\dots,\ov{\bx}}_{d-1},\underbrace{\bx,\dots,\bx}_d)- \mu \bx=0\\
d \cdot \F(\underbrace{\ov{\bx},\dots,\ov{\bx}}_{d},\bullet, \underbrace{\bx,\dots,\bx}_{d-1})- \mu \ov{\bx}=0\\
\bx^{\HH}\bx=1.
\end{array}
\right.
$$
The conclusion follows immediately by comparing the above with~\eqref{eq:Ceigen} and~\eqref{eq:Ceigen2}.
\end{proof}

Similarly, we have the following characterization. 
\begin{proposition}\label{prop:G-eigen}
$\bx\in\C^n$ is a G-eigenvector associated with a G-eigenvalue $\lambda\in\R$ for a conjugate super-symmetric tensor $\F$ if and only if $\bx$ is a KKT point of the optimization problem
$$
\max_{\bx^{\HH}\bx=1} \F\bigg( \underbrace{\dbinom{\ov\bx}{\bx}, \dots, \dbinom{\ov\bx}{\bx}}_d \bigg)
$$
with Lagrange multiplier being $d\lambda$ and the corresponding objective value being $\lambda$.
\end{proposition}

\subsection{Eigenvalues of complex tensors and their relations}

Although the definitions of the C-eigenvalue, the G-eigenvalue, and the previously defined Q-eigenvalue and the US-eigenvalue involve different tensor spaces, they are indeed closely related. Our main result in this section essentially states that the Q-eigenvalue and the US-eigenvalue are special cases of the C-eigenvalue, and the C-eigenvalue is a special case of the G-eigenvalue.

\begin{theorem}\label{thm:q-c-eigen}
Denote $\HI\in\C^{n^d}$ to be a complex tensor and define $\F=\ov\HI\otimes\HI\in\C^{n^{2d}}$. It holds that\\
(i) $\HI $ is symmetric if and only if $\F$ is conjugate partial-symmetric;\\
(ii) If $\HI$ is symmetric, then all the C-eigenvalues of $\F$ are nonnegative;\\
(iii) If $\HI$ is symmetric, then $\lambda^2$ is a C-eigenvalue of $\F$ if and only if $\lambda$ is a Q-eigenvalue (or a US-eigenvalue) of $\HI$.
\end{theorem}
\begin{proof}
(i) This equivalence can be easily verified by the definition of conjugate partial-symmetricity (Definition~\ref{thm:cps}).

(ii) Let $\bx\in\C^n$ be a C-eigenvector associated with a C-eigenvalue $\lambda\in\R$ of $\F$. By multiplying $\bx$ on both sides of the first equation in~\eqref{eq:Ceigen}, we obtain
\begin{eqnarray*}
\lambda &=& \F(\underbrace{\ov\bx,\dots,\ov\bx}_d,\underbrace{\bx,\dots,\bx}_d)
= (\ov\HI\otimes\HI)(\underbrace{\ov\bx,\dots,\ov\bx}_d,\underbrace{\bx,\dots,\bx}_d)
= \ov\HI(\underbrace{\ov\bx,\dots,\ov\bx}_d)\cdot\HI(\underbrace{\bx,\dots,\bx}_d) \\
&=& |\HI(\underbrace{\bx,\dots,\bx}_d)|^2 \ge 0.
\end{eqnarray*}

(iii) Since the Q-eigenvalue is the same as the US-eigenvalue, we only prove the former case. Suppose $\bx\in\C^n$ is a Q-eigenvector associated with a Q-eigenvalue $\lambda\in\R$ of $\HI$. By~\eqref{eq:Qeigen} we have $\bx^{\HH}\bx=1$ and $\HI(\bullet,\bx,\dots,\bx)=\lambda\ov\bx$, and so
$\HI(\underbrace{\bx,\dots,\bx}_d)=\lambda\bx^{\T}\ov\bx=\lambda.$
By the similar derivation in the proof of (ii), we get
\begin{align*}
\F(\bullet,\underbrace{\ov\bx,\dots,\ov\bx}_{d-1},\underbrace{\bx,\dots,\bx}_d) &= \ov\HI(\bullet,\underbrace{\ov\bx,\dots,\ov\bx}_{d-1})\cdot\HI(\underbrace{\bx,\dots,\bx}_d) \\
&=\ov{\HI(\bullet,\underbrace{\bx,\dots,\bx}_{d-1})}\cdot\lambda
= \ov{\lambda\ov\bx}\cdot\lambda=\lambda^2\bx,
\end{align*}
implying that $\lambda^2$ is a C-eigenvalue of $\F$.

On the other hand, suppose $\bx\in\C^n$ is a C-eigenvector associated with a nonnegative C-eigenvalue $\lambda^2$ of $\F$. Then by~\eqref{eq:Ceigen2} we have $\bx^{\HH}\bx=1$ and
\begin{align}
\ov{\HI(\underbrace{\bx,\dots,\bx}_d)} \cdot \HI(\bullet,\underbrace{\bx,\dots,\bx}_{d-1})
&= \ov\HI(\underbrace{\ov\bx,\dots,\ov\bx}_d) \cdot \HI(\underbrace{\bx,\dots,\bx}_{d-1},\bullet) \nonumber \\
&= \F(\underbrace{\ov\bx,\dots,\ov\bx}_d,\underbrace{\bx,\dots,\bx}_{d-1},\bullet)
= \lambda^2\ov\bx, \label{eq:Qeigen-Ceigen}
\end{align}
where the first equality is due to the symmetricity of $\HI$. This leads to $|\HI(\bx,\dots,\bx)|^2 =\lambda^2$. Let $\HI(\bx,\dots,\bx)=\lambda e^{\ii\theta}$ with {some fixed} $\theta\in[0,2\pi)$, and further define $\by=\bx e^{-\ii\theta/d}$. We then get
$$
\HI(\underbrace{\by,\dots,\by}_d)=\HI(\underbrace{\bx e^{-\ii\theta/d},\dots,\bx e^{-\ii\theta/d}}_d) = (e^{-\ii\theta/d})^d\HI(\underbrace{\bx,\dots,\bx}_d)
=e^{-\ii\theta}\lambda e^{\ii\theta}=\lambda.
$$
Now we are able to verify that $\by$ is a Q-eigenvector associated with Q-eigenvalue $\lambda$ of $\HI$.
Observing $\by^{\HH}\by=(\bx e^{-\ii\theta/d})^{\HH}\bx e^{-\ii\theta/d}=1$, and by~\eqref{eq:Qeigen-Ceigen},
\begin{align*}
\lambda^2\ov\bx&=\ov{\HI(\underbrace{\bx,\dots,\bx}_d)} \cdot \HI(\bullet,\underbrace{\bx,\dots,\bx}_{d-1})\\
&=\ov{\lambda e^{\ii\theta}}\HI(\bullet, \underbrace{\by e^{\ii\theta/d},\dots,\by e^{\ii\theta/d}}_{d-1})\\
&=\lambda e^{-\ii\theta}(e^{\ii\theta/d})^{d-1}\HI(\bullet,\underbrace{\by,\dots,\by}_{d-1}),
\end{align*}
we finally get
$$\HI(\bullet,\underbrace{\by,\dots,\by}_{d-1})=\lambda \ov{\bx}e^{\ii\theta/d}=\lambda\ov{\by e^{\ii\theta/d}} e^{\ii\theta/d}=\lambda\ov\by.$$
\end{proof}

{As we saw} in Section~\ref{sec:preparation}, by definition, a symmetric conjugate form is a special general conjugate form. Hence in terms of their tensor representations, a conjugate partial-symmetric tensor is a special case of conjugate super-symmetric tensor, although they live in different tensor spaces. To study the relationship between the C-eigenvalue and the G-eigenvalue, let us { embed a conjugate partial-symmetric tensor} $\F\in\C^{n^{2d}}$ to the space of $\C^{(2n)^{2d}}$. The conjugate super-symmetric tensor $\G\in\C^{(2n)^{2d}}$ corresponding to $\F$ is then defined by
\begin{equation}\label{eq:cps-css}
\G_{j_1\dots j_{2d}}=\left\{
\begin{array}{ll}
\F_{i_1\dots i_{2d}}/\binom{2d}{d}, & (j_1\dots j_{2d}) \in \Pi(i_1,\dots, i_d,i_{d+1}+n,\dots,i_{2d}+n); \\
0, & \mbox{otherwise}.
\end{array}
\right.
\end{equation}
For example when $d=1$, a conjugate partial-symmetric tensor is simply a Hermitian matrix $A \in \C^{n^2}$. Then its embedded conjugate super-symmetric tensor is
$\left( \begin{smallmatrix} O & A/2 \\A^{\T}/2 & O \end{smallmatrix}  \right)\in\C^{(2n)^2}$, and clearly we have
 $$
 \ov{\bx}^{\T}A\bx=\dbinom{\ov{\bx}}{\bx}^{\T}\left(
   \begin{array}{cc}
     O & A/2 \\
     A^{\T}/2 & O \\
   \end{array}
 \right)\dbinom{\ov{\bx}}{\bx}.
 $$
In general it is straightforward to verify that
\begin{equation}\label{eq:CG-link}
\F(\underbrace{\ov{\bx},\dots,\ov{\bx}}_d,\underbrace{\bx,\dots,\bx}_d) =\G\bigg( \underbrace{\dbinom{\ov\bx}{\bx}, \dots, \dbinom{\ov\bx}{\bx}}_{2d} \bigg).
\end{equation}
Based on this, we are led to the following relationship between the C-eigenvalue and the G-eigenvalue.
\begin{theorem}\label{thm:c-g-eigen}
If $\G\in \C^{(2n)^{2d}}$ is a conjugate super-symmetric tensor induced by a conjugate partial-symmetric tensor $\F \in \C^{n^{2d}}$ according to~\eqref{eq:cps-css}, then $\lambda$ is a C-eigenvalue of $\F$ if and only if $\lambda/2$ is a G-eigenvalue of $\G$.
\end{theorem}
\begin{proof}
First, by taking the gradient $\left(\frac{\partial}{\partial \ov\bx},\frac{\partial}{\partial \bx} \right)$ on both sides of~\eqref{eq:CG-link}, we have that
  $$
    \dbinom{d \cdot \F(\bullet,\overbrace{\ov{\bx},\dots,\ov{\bx}}^{d-1},\overbrace{\bx,\dots,\bx}^d)}
    {d \cdot \F(\underbrace{\ov{\bx},\dots,\ov{\bx}}_d,\underbrace{\bx,\dots,\bx}_{d-1},\bullet) } =  2d \cdot \G\bigg(\dbinom{\bullet}{\bullet}, \underbrace{\dbinom{\ov\bx}{\bx}, \dots,
    \dbinom{\ov\bx}{\bx}}_{2d-1} \bigg).
  $$
Next, according to Definition~\ref{thm:Ceigen} and Proposition~\ref{thm:Ceigen2}, $\lambda$ is a C-eigenvalue of $\F$ if and only if there exists a vector $\bx\in\C^n$ such that
  $$
\left\{
\begin{array}{l}
\F(\bullet, \underbrace{\ov{\bx},\dots,\ov{\bx}}_{d-1},\underbrace{\bx,\dots,\bx}_d) = \lambda \bx\\
\F(\underbrace{\ov{\bx},\dots,\ov{\bx}}_d, \underbrace{\bx,\dots,\bx}_{d-1}, \bullet) = \lambda \ov{\bx}\\
\bx^{\HH}\bx=1.
\end{array}
\right.
  $$
Finally, according to Definition~\ref{thm:Geigen}, $\lambda/2$ is a G-eigenvalue of $\G$ if and only if there exists a vector $\bx\in\C^n$ such that
$$
\left\{
\begin{array}{l} \G\bigg(\dbinom{\bullet}{\bullet}, \underbrace{\dbinom{\ov\bx}{\bx}, \dots,
\dbinom{\ov\bx}{\bx}}_{2d-1} \bigg)=\frac{\lambda}{2} \dbinom{\bx}{\ov\bx}\\
\bx^{\HH}\bx=1.
 \end{array}\right.
$$
The conclusion follows immediately by combining the above three facts.
\end{proof}

\section{Extending Banach's theorem to the real-valued conjugate forms}\label{sec:Banach}

A classical result originally due to Banach~\cite{B38} states that if $\CL(\bx^1,\dots,\bx^d)$ is a continuous symmetric $d$-linear form, then
\begin{equation}\label{banach}
\sup\{ |\CL(\bx^1,\dots,\bx^d)| \mid \|\bx^1\|\le1, \dots, \|\bx^d\|\le1\} = \sup \{ |\CL(\underbrace{\bx, \dots, \bx}_d)| \mid \|\bx\|\le1\}.
\end{equation}
In the space of real tensors where $\bx\in\R^n$ and $\CL$ is a multilinear form defined by a real symmetric tensor $\CL\in\R^{n^d}$,~\eqref{banach} states that the largest singular value~\cite{L05} of $\CL$ is equal to the {largest eigenvalue~\cite{Q05} (in the absolute value sense) of $\CL$ }, i.e.,
\begin{equation}\label{banach_real}
  \max_{(\bx^k)^{\T}\bx^k= 1,\,\bx^k \in \R^n,\,k=1,\dots,d}  \CL( \bx^1, \dots, \bx^d) = \max_{\bx^{\T}\bx=1,\,\bx \in \R^n} | \CL( \underbrace{\bx, \dots, \bx}_d ) |.
\end{equation}
Alternatively,~\eqref{banach_real} is essentially equivalent to the fact that the best rank-one approximation of a real symmetric tensor can be obtained at a symmetric rank-one tensor~\cite{CHLZ12,ZLQ12}. A recent development on this topic for special classes of real symmetric tensors can be found in~\cite{CHLZ14}. In this section, we shall extend Banach's theorem to { symmetric conjugate forms (the conjugate partial-symmetric tensors) and general conjugate} forms (the conjugate super-symmetric tensors).

\subsection{Equivalence for conjugate super-symmetric tensors} \label{sec:gbanach}

Let us start with conjugate super-symmetric tensors, which are a generalization of conjugate partial-symmetric tensors. A key observation leading to the equivalence (Theorem~\ref{thm:equal-con}) is the following result.
\begin{lemma}\label{lemma:super-sym}
For a given real tensor $\F\in \R^{n^d}$, if $\F(\bx^1, \dots, \bx^d) =  \F( \bx^{\pi(1)}, \dots, \bx^{\pi(d)})$ for every $\bx^1,\dots,\bx^d \in \R^n$ and every permutation $\pi$ of $\{1,\dots,d \}$, then $\F$ is symmetric.
\end{lemma}

Our first result in this section extends~\eqref{banach_real} to any conjugate super-symmetric tensors in the complex domain.
\begin{theorem}\label{thm:equal-con}
For any conjugate super-symmetric tensor $\G\in\C^{(2n)^d}$, we have
\begin{equation}\label{eq:equal-conjugate}
\max_{\bx^{\HH}\bx=1} \bigg{|} \G\bigg( \underbrace{\dbinom{\ov\bx}{\bx}, \dots, \dbinom{\ov\bx}{\bx}}_d \bigg) \bigg{|}  = \max_{(\bx^k)^{\HH}\bx^k=1,\,k=1,\dots,d}\re  \G\bigg( \dbinom{\,\ov{\bx^1}\,}{\bx^1}, \dots, \dbinom{\,\ov{\bx^d}\,}{\bx^d} \bigg).
\end{equation}
\end{theorem}
\begin{proof} Let $\by^k = \binom{\re \bx^k}{\im \bx^k}\in\R^{2n}$ for $k=1,\dots, d$. We observe that $\re \G\big( \binom{\,\ov{\bx^1}\,}{\bx^1}, \dots, \binom{\,\ov{\bx^d\,}}{\bx^d} \big)$ is also a multilinear form with respect to $\by^1,\dots,\by^d$. As a result, we are able to find a real tensor $\F\in \R^{(2n)^d}$ such that
\begin{equation}\label{eq:real-conjugate}
\F\big( \by^1, \dots, \by^d\big) = \re  \G\bigg(\,\dbinom{\,\ov{\bx^1}\,}{\bx^1}, \dots, \dbinom{\,\ov{\bx^d}\,}{\bx^d} \bigg).
\end{equation}
As $\G$ is conjugate super-symmetric, for any $\by^1,\dots,\by^d \in \R^{2n}$ and any permutation $\pi$ of $\{1,\dots,d\}$, one has
\begin{align*}
\F\big( \by^1, \dots, \by^d\big) &= \re  \G\bigg( \dbinom{\,\ov{\bx^1}\,}{\bx^1}, \dots, \dbinom{\,\ov{\bx^d}\,}{\bx^d} \bigg) \\
&= \re  \G\bigg( \dbinom{\,\ov{\bx^{\pi(1)}}\,}{\bx^{\pi(1)}}, \dots, \dbinom{\,\ov{\bx^{\pi(d)}}\,}{\bx^{\pi(d)}} \bigg)\\
&= \F\big( \by^{\pi(1)}, \dots, \by^{\pi(d)}\big).
\end{align*}
By Lemma~\ref{lemma:super-sym} we have that the real tensor $\F$ is symmetric. Finally, noticing that $(\by^k)^{\T}\by^k=(\bx^k)^{\HH}\bx^k$ for $k=1,\dots,d$, the conclusion follows immediately by applying~\eqref{banach_real} to $\F$ and then using the equality~\eqref{eq:real-conjugate}.
\end{proof}

\subsection{Equivalence for conjugate partial-symmetric tensors} \label{sec:cbanach}

For {extending Banach's theorem to a conjugate partial-symmetric tensor $\F\in\C^{n^{2d}}$, one could hope to proceed as follows. Since it is} a special case of the conjugate super-symmetric tensor, one can embed $\F$ into a conjugate super-symmetric tensor
 $\G\in\C^{(2n)^{2d}}$ using~\eqref{eq:cps-css}. Then, by applying Theorem~\ref{thm:equal-con} to $\G$ and rewriting the real part of its associated multilinear form $\re\G\big( \binom{\,\ov{\bx^1}\,}{\bx^1}, \dots, \binom{\,\ov{\bx^{2d}}\,}{\bx^{2d}} \big)$ in terms of $\F$, we may have an equivalent expression as~\eqref{eq:equal-conjugate}. However, this expression is not succinct.
Taking the case $d=2$ (degree 4) for example, it is straightforward to verify from~\eqref{eq:cps-css} that
\begin{align*}
&~~~~\re\G\bigg(\binom{\,\ov{\bx^1}\,}{\bx^1},\binom{\,\ov{\bx^2}\,}{\bx^2}, \binom{\,\ov{\bx^3}\,}{\bx^3},\binom{\,\ov{\bx^4}\,}{\bx^4}\bigg)\\
&=\frac{1}{6} \left(\F(\ov{\bx^1},\ov{\bx^2},\bx^3,\bx^4)
+\F(\ov{\bx^1},\ov{\bx^3},\bx^2,\bx^4)
+\F(\ov{\bx^1},\ov{\bx^4},\bx^2,\bx^3)\right.\\
&~~~\left.+\F(\ov{\bx^2},\ov{\bx^3},\bx^1,\bx^4)
+\F(\ov{\bx^2},\ov{\bx^4},\bx^1,\bx^3)
+\F(\ov{\bx^3},\ov{\bx^4},\bx^1,\bx^2)\right)\\
&:=f_S'(\bx^1,\bx^2,\bx^3,\bx^4),
\end{align*}
and this would lead to
$$
\max_{\bx^{\HH}\bx=1} | \F(\ov{\bx},\ov{\bx},\bx,\bx) |  = \max_{(\bx^k)^{\HH}\bx^k=1,\,k=1,2,3,4}f_S'(\bx^1,\bx^2,\bx^3,\bx^4).
$$
Instead, one would hope to get
\begin{equation}\label{eq:equivalence-conjugate}
\max_{\bx^{\HH}\bx=1} | \F(\underbrace{\ov{\bx},\dots,\ov{\bx}}_d,\underbrace{\bx,\dots,\bx}_d) | = \max_{(\bx^k)^{\HH}\bx^k=1,\,k=1,\dots,2d} \re\F(\ov{\bx^1},\dots,\ov{\bx^d},\bx^{d+1},\dots,\bx^{2d}).
\end{equation}
However, this does not hold in general. The main reason is that
$$\G\bigg( \dbinom{\,\ov{\bx^1}\,}{\bx^1}, \dots, \dbinom{\,\ov{\bx^{2d}}\,}{\bx^{2d}} \bigg)\neq \F(\ov{\bx^1},\dots,\ov{\bx^d},\bx^{d+1},\dots,\bx^{2d}),$$
which is easily observed since its left hand side is invariant under the permutation of $(\bx^1,\dots,\bx^{2d})$ while its right hand side is not. In particular,~\eqref{eq:equivalence-conjugate} only holds for $d=1$, viz.\ Hermitian matrices; {see the following proposition} and Example~\ref{ex:quarticfail}.

\begin{proposition}\label{degree2}
  For any Hermitian matrix $Q\in\C^{n\times n}$, it holds that
   $$
     (L)\quad {\max_{\bz^{\HH}\bz=1}  \bz^{\HH} Q \bz } = \max_{\bx^{\HH}\bx=\by^{\HH}\by=1} \re \bx^{\T} Q \by. \quad (R)
   $$
   Furthermore, for any optimal solution $(\bx^*,\by^*)$ of $(R)$ with $\ov{\bx^*}+\by^*\neq 0$, $(\ov{\bx^*} + \by^*)/\|\ov{\bx^*} + \by^*\|$ is an optimal solution of $(L)$ as well.
\end{proposition}
\begin{proof}
Denote $v(L)$ and $v(R)$ to be the optimal values of $(L)$ and $(R)$, respectively. Noticing that $\re \bx^{\T} Q \by = \frac{1}{2}(\bx^{\T} Q \by + \ov \bx^{\T} \ov Q \ov \by)$, by the optimality condition of $(R)$ we have that
\begin{equation}\label{eq:2opt-condition-CQP}
\left\{
\begin{array}{l} Q \by^* - 2\lambda \ov {\bx^*} =0\\
\ov Q \ov {\by^*} - 2\lambda \bx^* =0\\
\ov Q \bx^* - 2\mu \ov {\by^*} =0\\
Q \ov {\bx^*} - 2\mu \by^* =0\\
(\bx^*)^{\HH}\bx^*=1\\
(\by^*)^{\HH}\by^*=1,
\end{array}
\right.
\end{equation}
where $\lambda$ and $\mu$ are the Lagrangian multipliers of the constraints $\bx^{\HH}\bx=1$ and $\by^{\HH}\by=1$, respectively.

{Pre-multiplying the first two equations in~\eqref{eq:2opt-condition-CQP} with $(\bx^*)^{\T}$ and $(\ov{\bx^*})^{\T}$ respectively, and summing them up,} lead to
$$2\re (\bx^*)^{\T} Q \by^* = (\bx^*)^{\T} Q \by^* + (\ov{\bx^*})^{\T} \ov Q \ov{\by^*} = 2 \lambda (\bx^*)^{\T}\ov{\bx^*} + 2 \lambda (\ov{\bx^*})^{\T}\bx^* = 4\lambda(\bx^*)^{\HH}\bx^* = 4\lambda.$$
Similarly, the summation of the third and fourth equations in~\eqref{eq:2opt-condition-CQP} leads to
$$2\re (\bx^*)^{\T} Q \by^* = 4\mu,$$
which further leads to
\begin{equation}\label{eq:CQP}
v(R)= \re (\bx^*)^{\T} Q \by^* = 2\lambda =2 \mu.
\end{equation}

Moreover, the summation of the first and {fourth} equations in~\eqref{eq:2opt-condition-CQP} yields
$$Q (\by^* + \ov{\bx^*}) - 2\lambda (\ov {\bx^*} + \by^*)=0,$$
which further leads to
\begin{equation}\label{eq:CQP2}
(\by^* + \ov{\bx^*})^{\HH}Q (\by^* + \ov{\bx^*}) =2\lambda (\by^* + \ov{\bx^*})^{\HH}(\ov {\bx^*} + \by^*)= 2\lambda\|\ov {\bx^*} + \by^*\|^2.
\end{equation}
Let $\bz^*=(\ov {\bx^*} + \by^*)/\|\ov {\bx^*} + \by^*\|$. Clearly $\bz^*$ is a feasible solution of $(L)$. {By~\eqref{eq:CQP2} and~\eqref{eq:CQP} we have}
$$(\bz^*)^{\HH}Q\bz^*=2\lambda=\re (\bx^*)^{\T} Q \by^*=v(R).$$
This implies that $v(L)\ge v(R)$. Notice that $(R)$ is a relaxation of $(L)$ and hence $v(L)\le v(R)$. Therefore we conclude that $v(R)=v(L)$, and an optimal solution $\bz^*$ of $(L)$ is constructed from an optimal solution $(\bx^*,\by^*)$ of $(R)$.
\end{proof}

\begin{example}\label{ex:quarticfail}
  Let $\F\in\C^{2^4}$ with $\F_{1122}=\F_{2211}=1$ and other entries being zeros. Clearly $\F$ is conjugate partial-symmetric. In this case~\eqref{eq:equivalence-conjugate} fails to hold because:\\
  (i) $|\F(\ov{\bx},\ov{\bx},\bx,\bx)|=|{\ov{x_1}}^2{x_2}^2+{\ov{x_2}}^2{x_1}^2| \le 2|x_1|^2|x_2|^2\le\frac{1}{2}(|x_1|^2+|x_2|^2)^2=\frac{1}{2}$ for any $\bx\in\C^2$ with $\bx^{\HH}\bx=1$.\\
  (ii) $\F(\ov{\bx},\ov{\by},\bz,\bw)=\ov{x_1}\ov{y_1}z_2w_2+\ov{x_2}\ov{y_2}z_1w_1=1$ for $\bx=\by=(1,0)^{\T}$ and $\bz=\bw=(0,1)^{\T}$.
\end{example}

Thus, Banach's theorem~\eqref{eq:equivalence-conjugate} does not hold in general for conjugate partial-symmetric tensors. A natural question arises: Is there any reasonable condition to ensure the identity to hold? 
Recall from Proposition~\ref{thm:conju-decomp} that every conjugate partial-symmetric tensor can be written as $\sum_{k=1}^m\alpha_k\ov{\HI_k}\otimes \HI_k$ where $\HI_k\in\C^{n^d}$ is symmetric and $\alpha_k\in\R$ for all $1\le k\le m$. If further we have all $\alpha_k$'s being nonnegative, then~\eqref{eq:equivalence-conjugate} is true. Before presenting this result, we first need the following type of Banach's theorem for symmetric complex tensors, whose proof can be constructed almost identically to that of Theorem~\ref{thm:equal-con}.

\begin{proposition} \label{thm:complexBanach}
  If $F\in\C^{n^d}$ is symmetric, then
\begin{equation}
\max_{\bx^{\HH}\bx=1} \re\F(\underbrace{\bx,\dots,\bx}_d) = \max_{(\bx^k)^{\HH}\bx^k=1,\,k=1,\dots,d} \re\F(\bx^1,\dots,\bx^d).
\end{equation}
\end{proposition}
\begin{theorem}\label{thm:Banach2}
If a conjugate partial-symmetric tensor $\F\in\C^{n^{2d}}$ written as $\sum_{k=1}^m\alpha_k \ov{\HI_k}\otimes \HI_k$ satisfies that $\alpha_k\ge0$ for all $1\le k\le m$, then
$$
(L')\quad\max_{\bx^{\HH}\bx=1} \F(\underbrace{\ov{\bx},\dots,\ov{\bx}}_d, \underbrace{\bx,\dots,\bx}_d) = \max_{(\bx^k)^{\HH}\bx^k=1,\,k=1,\dots,2d} \re\F(\bx^1,\dots,\bx^{2d}) \quad (R')
$$
\end{theorem}
\begin{proof}
Let us first introduce a sandwiched optimization model:
$$
(M')\quad \max_{\by^{\HH}\by=\bz^{\HH}\bz=1} \re\F(\underbrace{\by,\dots,\by}_d, \underbrace{\bz,\dots,\bz}_d).
$$
Denote $v(L')$, $v(M')$ and $v(R')$ to be the optimal values of $(L')$, $(M')$ and $(R')$, respectively. Clearly, $(R')$ is a relaxation of $(M')$, and $(M')$ is a relaxation of $(L')$, implying that $v(L')\le v(M') \le v(R')$.

Next, let $(\bx^1_*,\dots,\bx^{2d}_*)$ be an optimal solution of $(R')$. Consider the following problem:
$$
\max_{\by^{\HH}\by=1} \re\F(\underbrace{\by,\dots,\by}_d, \bx^{d+1}_*,\dots,\bx^{2d}_*),
$$
whose optimal solution is denoted by $\by^*$. Noticing that $\F(\bullet,\dots,\bullet, \bx^{d+1}_*,\dots,\bx^{2d}_*)\in\C^d$ is symmetric, by Proposition~\ref{thm:complexBanach}, we have
\begin{align*}
\re\F(\underbrace{\by^*,\dots,\by^*}_d, \bx^{d+1}_*,\dots,\bx^{2d}_*) &=\max_{\by^{\HH}\by=1} \re\F(\underbrace{\by,\dots,\by}_d, \bx^{d+1}_*,\dots,\bx^{2d}_*) \\
&= \max_{(\bx^k)^{\HH}\bx^k=1,\,k=1,\dots,d} \re \F(\bx^1,\dots,\bx^d, \bx^{d+1}_*,\dots,\bx^{2d}_*)\\
&\ge \re \F(\bx^1_*,\dots,\bx^{2d}_*)=v(R').
\end{align*}
For the same reason, we have
\begin{align*}
\max_{\bz^{\HH}\bz=1} \re\F(\underbrace{\by^*,\dots,\by^*}_d, \underbrace{\bz,\dots,\bz}_d) & = \max_{(\bx^k)^{\HH}\bx^k=1,\,k=d+1,\dots,2d} \re\F(\underbrace{\by^*,\dots,\by^*}_d, \bx^{d+1},\dots,\bx^{2d}) \\
&\ge \re\F(\underbrace{\by^*,\dots,\by^*}_d, \bx^{d+1}_*,\dots,\bx^{2d}_*) \ge v(R'),
\end{align*}
implying that $v(M')\ge v(R')$.

Finally, let $(\by^*,\bz^*)$ be an optimal solution of $(M')$. Since $\alpha_k\ge0$ for all $0\le k\le m$, we have
\begin{align*}
  &~~~~\re\F(\underbrace{\by^*,\dots,\by^*}_d, \underbrace{\bz^*,\dots,\bz^*}_d)\\
  &= \re \left(\sum_{k=1}^m\alpha_k \ov{\HI_k}\otimes \HI_k\right) (\underbrace{\by^*,\dots,\by^*}_d, \underbrace{\bz^*,\dots,\bz^*}_d)\\
  &= \sum_{k=1}^m \alpha_k \re (\ov{\HI_k}(\underbrace{\by^*,\dots,\by^*}_d)\cdot \HI_k(\underbrace{\bz^*,\dots,\bz^*}_d)) \\
  &\le  \sum_{k=1}^m \alpha_k |\ov{\HI_k}(\underbrace{\by^*,\dots,\by^*}_d) |\cdot| \HI_k(\underbrace{\bz^*,\dots,\bz^*}_d)| \\
  &\le  \sum_{k=1}^m \frac{\alpha_k}{2}\bigg(|\ov{\HI_k}(\underbrace{\by^*,\dots,\by^*}_d)|^2 + |\HI_k(\underbrace{\bz^*,\dots,\bz^*}_d)|^2\bigg) \\
  &=  \frac{1}{2}\sum_{k=1}^m \alpha_k\bigg((\ov{\HI_k}\otimes \HI_k)(\underbrace{\by^*,\dots,\by^*}_d,\underbrace{\ov{\by^*},\dots,\ov{\by^*}}_d) + (\ov{\HI_k}\otimes \HI_k)(\underbrace{\bz^*,\dots,\bz^*}_d, \underbrace{\ov{\bz^*},\dots,\ov{\bz^*}}_d)\bigg) \\
  &= \frac{1}{2}\bigg(\F(\underbrace{\by^*,\dots,\by^*}_d, \underbrace{\ov{\by^*},\dots,\ov{\by^*}}_d) + \F( \underbrace{\bz^*,\dots,\bz^*}_d, \underbrace{\ov{\bz^*},\dots,\ov{\bz^*}}_d)\bigg) \\
  &\le \max\bigg\{ \F(\underbrace{\by^*,\dots,\by^*}_d, \underbrace{\ov{\by^*},\dots,\ov{\by^*}}_d), \F(\underbrace{\bz^*,\dots,\bz^*}_d, \underbrace{\ov{\bz^*},\dots,\ov{\bz^*}}_d)\bigg\}.
\end{align*}
{Remark that the positivity of $\alpha_k$'s is exploited when invoking the triangle inequality in the first inequality above.} This implies that either $\F(\underbrace{\by^*,\dots,\by^*}_d, \underbrace{\ov{\by^*},\dots,\ov{\by^*}}_d)$ or $\F(\underbrace{\bz^*,\dots,\bz^*}_d, \underbrace{\ov{\bz^*},\dots,\ov{\bz^*}}_d)$ attains $v(M')$, proving that $v(L')\ge v(M')$. Therefore we have $v(L')= v(M') = v(R')$.
\end{proof}

We remark that the condition for $\alpha_k$'s being nonnegative in $\F$ in Theorem~\ref{thm:Banach2} is actually the condition for the real-valued symmetric conjugate form $\mathbf{S}(\F)$ being a sum of squares (SOS) of complex polynomials; see the relation between {Propositions~\ref{thm:polyrealvalued} and~\ref{thm:conju-decomp}}. In the field of polynomial optimization, checking whether a polynomial is SOS can be done by the feasibility of a semidefinite program. In fact, there is an easy sufficient condition for the condition {on} $\F$ in Theorem~\ref{thm:Banach2} to hold: the square matrix flattening of $\F$ is Hermitian positive semidefinite.
Interested readers are referred to~\cite{JLZ15} for details.

\section{Applications} \label{sec:application}

The theoretical results developed {in the previous sections are also useful in practice. In this section, we shall discuss some applications that can be formulated as real-valued complex polynomial optimization models. In particular, these problems can be cast as finding the largest C-eigenvalue of a conjugate partial-symmetric tensor or the largest G-eigenvalue of a conjugate super-symmetric tensor.}

One challenge of these eigenvalue optimization problems is that the variables are coupled in the { complex polynomial objective function}. However, the extended Banach's theorem in Section~\ref{sec:Banach}, specifically Theorems~\ref{thm:equal-con} and~\ref{thm:Banach2}, guarantee that we can separate the variables without losing the optimality. This enables us to focus on the multilinear (block) optimization model
$$
\max_{(\bx^k)^{\HH}\bx^k=1,\,k=1,\dots,d}\re  \G\bigg( \binom{\,\ov{\bx^1}\,}{\bx^1}, \dots, \binom{\,\ov{\bx^d}\,}{\bx^d} \bigg)
$$
for a conjugate super-symmetric tensor $\G$, or
$$
\max_{(\bx^k)^{\HH}\bx^k=1,\,k=1,\dots,2d} \re\F(\bx^1,\dots,\bx^{2d})
$$
for certain conjugate partial-symmetric tensor $\F$. One great advantage of the above models is that the optimization over one block variable is easy when other blocks are fixed. Therefore, some efficient solution methods tailored for these models can be applied, such as the block coordinate decent method~\cite{LT92} and the maximum block improvement method~\cite{CHLZ12}. {Conversely, the extended Banach's theorem in Section~\ref{sec:Banach} provides an alternative way to solve the symmetric multilinear optimization model by resorting to some approaches tailored for symmetric tensor problems such as the power method~\cite{KoldaMayo11} and the semidefinite programming method~\cite{NieWang14,JMZ15}. In particular, as the search space can be restricted to symmetric solutions, the latter equivalent model significantly reduces the number of decision variables, which is beneficial to many practical algorithms such as semidefinite programs.}

\subsection{Ambiguity function shaping for radar waveform}

The ambiguity function of the waveform is often used to probe the environment in radar system. By controlling both the Doppler and the range resolutions of the system, it can regulate the interference power produced by unwanted returns~\cite{ADJZ12}. To be specific, suppose $v_0$ is the normalized target Doppler frequency and $\bs = (s_1,\dots,s_n)^{\T} \in \C^n$ is the radar code to be optimized. There are $n_0$ interfering scatterers and the matrix $J^r\in\R^{n^2}$ for $r\in\{0,1,\dots,n-1\}$ is defined as
\[
{(J^r)}_{ij}=\left\{
\begin{array}{cc}
1 &  i-j=r \\
0 &  i-j\ne r
\end{array}\right., \quad\forall\, 1\le i, j\le n.
\]
The ambiguity function of $\bs$ for the time-lag $r\in\{0,1,\dots,n-1\}$ and the normalized Doppler frequency $v\in \left[-\frac{1}{2}, \frac{1}{2}\right]$ is given by
$$
g_{\bs}(r,v)=\frac{1}{\|\bs\|^2}\left|\bs^{\HH}J^r(\bs\odot \bp(v))\right|^2,
$$
where $\bp(v) = (1, e^{\ii2\pi v} ,\dots , e^{\ii2(n-1)\pi v})^{\T}$ and $\odot$ denotes the Hadamard product; interested readers are referred to~\cite{ADJZ12} for more details of the ambiguity function and radar waveform design.

{Denote by $r_k$ the time-lag of the $k$-th scatterer, and let $v_k$ be} the normalized Doppler frequency of the $k$-th scatterer. The latter is usually modeled as a random variable uniformly distributed around {a mean frequency $\hat{v}_k$ with some tolerance $\frac{\epsilon_k}{2}$}, i.e., $v_k$ is a uniform distribution in $\left[\hat{v}_k - \frac{\epsilon_k}{2}, \hat{v}_k + \frac{\epsilon_k}{2}\right]$. Consequently, the disturbance power at the output of the matched filter is given by
\begin{equation}\label{continous-distur-power}
\sum_{k=1}^{n_0}{\sigma_k}^2\|\bs\|^2 \ex [g_{\bs} (r_k, v_{k} -v_{0} )] + \sigma^2\|\bs\|^2,
\end{equation}
where $\sigma^2$ is the variance of the circular white noise, and ${\sigma_k}^2$ is the echo mean power produced by the $k$-th scatterer. To simplify the notation, all the following normalized Doppler frequencies are expressed in terms of the difference with respect to $v_0$. We discretize the normalized Doppler interval $[-\frac{1}{2}, \frac{1}{2})$ into $m$ bins, denoted by discrete frequencies $x_j = -\frac{1}{2} + \frac{j}{m}$ for $j\in\{0,1,\dots,m\}$. Let
$$
\Delta_k = \left\{ j:\left[x_j - \frac{1}{2m},x_j + \frac{1}{2m}\right) \bigcap \left[\hat{v}_k - \frac{\epsilon_k}{2}, \hat{v}_k + \frac{\epsilon_k}{2}\right] \neq \varnothing \right\}.
$$
Then the above statistical expectations can be approximated by the sample means over $\Delta_k$, i.e.,
$$
\ex [g_{\bs} (r_k, v_{k} )] \approx \frac{1}{|\Delta_k|}\sum_{j \in \Delta_k} g_{\bs} (r_k, x_{j} ),
$$
Plugging the above expression into~\eqref{continous-distur-power}, the total disturbance power at the output of the matched filter can be rewritten as
$$
\phi(\bs)=\sum_{r=0}^{n-1} \sum_{j=1}^{m} \rho(r,k) |\bs^{\HH}J^r (\bs\odot\bp(x_j))|^2,
$$
where $\rho(r,k) = \sum_{k=1}^{n_0}{\delta_{r,r_k}}{\bf{1}}_{\Delta_k}(j)\frac{\sigma_k^2}{|\Delta_k|}$ with {$\delta_{r,r_k}$ being the Kronecker delta and} ${\bf{1}}_{\Delta_k}(j)$ being an indicator function.

To obtain phase-only modulated waveforms, an optimization model to minimize $\phi(\bs)$ subject to constant modulus constraints was proposed in~\cite{ADJZ12}:
$\min_{|s_i| = 1,\, i=1,\dots,n}\phi(\bs)$.
{Another} modeling strategy is to account for the finite energy transmitted by the radar and assume that $\|\bs\|^2 = 1$. However, this single constraint does not provide any kind of control on the shape of the resulting coded waveform. To circumvent this drawback, one practical approach 
is to enforce a similarity constraint (see~\cite{ADFW13} for more details):
\begin{equation}\label{similarity-constraint}
\| \bs - \bs^0 \|^2 \le \gamma,
\end{equation}
where $\bs^0$ is a known code which shares some nice properties {such as a constant modula and a reasonable range resolution}. Moreover, since any feasible $\bs$ satisfies $\|\bs\|=1$ and
$$
\| \bs - \bs^0 \|^2 = \| \bs \|^2 + \| \bs^0 \|^2 -  (\bs^{\HH}\bs^0 + (\bs^0)^{\HH}\bs) = 1 + \| \bs^0 \|^2 -  (\bs^{\HH}\bs^0 + (\bs^0)^{\HH}\bs).
$$
Therefore, $\| \bs - \bs^0 \|^2 \le \gamma$ is equivalent to $- (\bs^{\HH}\bs^0 + (\bs^0)^{\HH}\bs) \le \gamma - 1 - \| \bs^0 \|^2$. Typically, the similarity constraint~\eqref{similarity-constraint} is not a hard constraint, it aims to restrict the searching area within some neighborhood of $\bs^0$ and the size of the neighborhood is controlled by $\gamma$. Motivated by the aforementioned equivalence, {a similar result} can be achieved by penalizing the quantity $- (\bs^{\HH}\bs^0 + (\bs^0)^{\HH}\bs)$ in the objective and we arrive the following formulation
\begin{equation}\label{eq:radar}
  \min_{\|\bs\| = 1} \left(\phi(\bs) - \rho (\bs^{\HH}\bs^0 + (\bs^0)^{\HH}\bs)^2 \| \bs\|^2\right)
\end{equation}
with penalty parameter $\rho$. Notice that the objective function in~\eqref{eq:radar} is a real-valued quartic conjugate complex form. If $\bs^*$ is the optimal solution and so is $-\bs^*$, then we can choose one of them to make sure that ${(\bs^*)^{\HH}\bs^0 + (\bs^0)^{\HH}\bs^*}>0$. The model~\eqref{eq:radar} is obviously finding the smallest C-eigenvalue of a conjugate partial-symmetric tensor, which can also be viewed as 
finding the smallest G-eigenvalue of a conjugate super-symmetric tensor as mentioned in Theorem~\ref{thm:c-g-eigen}.

\subsection{The best rank-one approximation of a complex tensor}

Many modern engineering problems can be cast as multilinear least squares regression given as
\begin{equation}\label{eq:tensor1}
\min_{\bz^k\in\C^{n_k},\,k=1,\dots,d}\frac{1}{2}\|\bz^1\otimes\dots \otimes \bz^d - \F \|^2,
\end{equation}
where $\F \in \C^{n_1\times \dots \times n_d }$ is a given nonzero complex tensor. For instance, in quantum entanglement the geometric measure of a given $d$-partite pure state $\F$ is defined by~\eqref{eq:tensor1}; see~\cite{WG03,NQB14} for details.

In fact,~\eqref{eq:tensor1} can be also categorized as a G-eigenvalue problem for a conjugate super-symmetric tensor. To see this, first it is easy to see that~\eqref{eq:tensor1} is equivalent to
$$
\min_{\lambda\in\R,\,\|\bz^k\|=1,\,k=1,\dots,d}\|\lambda \bz^1\otimes\dots \otimes \bz^d  - \F \|^2.
$$
{When all $\bz^k$'s with $\|\bz^k\| =1$ for $k = 1,\dots, d$ are fixed, the optimal $\lambda$ satisfies}
\begin{align*}
    \min_{\lambda \in \R}\|\lambda \bz^1\otimes\dots \otimes \bz^d  - \F \|^2 &=\min_{\lambda \in \R}\left( \|\F\|^2 - 2 \lambda \re \F(\bz^1,\dots,\bz^d) + \lambda^2 \right) \\
    &=\|\F\|^2 - (\re\F(\bz^1,\dots, \bz^d))^2.
\end{align*}
Therefore, by multilinearity,~\eqref{eq:tensor1} is equivalent to
\begin{equation}\label{eq:tensor2}
\max_{\|\bz^k\|=1,\,k=1,\dots,d}|\re\F(\bz^1,\dots, \bz^d)|=\max_{\|\bz^k\|=1,\,k=1,\dots,d}\re\F(\bz^1,\dots, \bz^d).
\end{equation}

Let us now consider a relaxation of the above model
\begin{equation}\label{eq:tensor3}
\max_{\sum_{k=1}^d\|\bz^k\|^2 = d }\re \F(\bz^1, \dots , \bz^d).
\end{equation}
A key observation is that this relaxation is actually tight. To see this, suppose $(\bz_*^1,\dots, \bz_*^{d})$ is an optimal solution of~\eqref{eq:tensor3}. Trivially we have $\re\F(\bz_*^1,\bz_*^2,\dots,\bz_*^{d})>0$ as $\F$ is nonzero and so $\|\bz_*^k\|\ne 0$ for $k=1,\dots,d$.
By noticing
$$
\left(\prod_{k=1}^d\|\bz_*^k\|^2\right)^{1/d} \le \frac{1}{d}\sum_{k=1}^{d}\|\bz_*^k\|^2=1,
$$
we have that $\prod_{k=1}^d\|\bz_*^k\|\le1$ and so
$$
\re\F\left(\frac{\bz_*^1}{\|\bz_*^1 \|},\dots,\frac{\bz_*^{d}}{\|\bz_*^{d} \|}\right)=\re\frac{\F(\bz_*^1,\dots,\bz_*^{d})}{\prod_{k=1}^{d}\|\bz_*^k\|} \ge \re \F(\bz_*^1,\bz_*^2,\dots,\bz_*^{d}).
$$
Therefore, the feasible solution $\left(\bz_*^1/\|\bz_*^1 \|,\dots,\bz_*^d/\|\bz_*^d\|\right)$ of~\eqref{eq:tensor2} is already optimal to the relaxation model~\eqref{eq:tensor3}, proving the equivalence between~\eqref{eq:tensor2} and~\eqref{eq:tensor3}.

Finally, to formulate~\eqref{eq:tensor3} as a G-eigenvalue {optimization} problem, let us denote $\bz=\left((\bz^1)^{\T},\dots,(\bz^{d})^{\T}\right)^{\T}\in\C^{nd}$ and construct a symmetric complex tensor $\HI\in\C^{(nd)^d}$ such that
$$
\HI(\underbrace{\bz,\dots,\bz}_{d}) = \F(\bz^1,\dots,\bz^d).
$$
Thus,~\eqref{eq:tensor3} can be rewritten as
\begin{align*}
\max_{\|\bz\|=\sqrt{d}} \re \HI(\underbrace{\bz,\dots,\bz}_{d})
&=\max_{\|\bz\|=\sqrt{d}} \frac{1}{2} \left(\HI(\underbrace{\bz,\dots,\bz}_{d}) + \ov\HI(\underbrace{\ov\bz,\dots,\ov\bz}_{d})\right)\\
&=\max_{\|\bx\|=1} \frac{\sqrt{d^d}}{2} \left(\HI(\underbrace{\bx,\dots,\bx}_{d}) + \ov\HI(\underbrace{\ov\bx,\dots,\ov\bx}_{d})\right)\\
&=\max_{\|\bx\|=1} \G\bigg(\underbrace{\dbinom{\ov\bx}{\bx}, \dots,
\dbinom{\ov\bx}{\bx}}_d \bigg),
\end{align*}
where $\G\in\C^{(2nd)^d}$ is a conjugate super-symmetric tensor. The multilinear least square model~\eqref{eq:tensor1} is shown to be a special case of the G-eigenvalue optimization problem.

\section{Conclusion}\label{sec:conclusion}

This paper focuses on complex polynomial functions that incorporate conjugate variables. We {introduced} two types of conjugate complex forms and their symmetric tensor representations. Necessary and sufficient conditions for these conjugate complex forms being real-valued are presented, based on which two types of symmetric complex tensors are introduced. We present new definitions of eigenvalues/eigenvectors, namely the $C$-eigenvalue and the $G$-eigenvalue, which generalize the existing concepts of eigenvalues in the literature. Extensions of Banach~\cite{B38} type's theorem on these complex tensors are discussed as well. To give the readers a holistic picture 
Table~\ref{table} summarizes the main contents.

\begin{table}[h]
\begin{center}
\begin{tabular}{ | l  p{60mm} l |}
\hline
Sec.  &  Subject &  Results \\ \hline
\ref{sec:cform}    & Symmetric conjugate form {and} partial-symmetric tensor  &  Def.~\ref{def:sform}, Def.~\ref{def:partial-symmetric}, Lemma~\ref{thm:tensorS}  \\
\ref{sec:ctensor}  & Real-valued symmetric conjugate form {and} conjugate partial-symmetric tensor  & Cor.~\ref{thm:condition}, Def.~\ref{thm:cps}, Prop.~\ref{thm:conju-mapping} \\
\ref{sec:eigendef} & $C$-eigenvalue and $C$-eigenvector & Def.~\ref{thm:Ceigen}, Prop.~\ref{thm:Ceigen2}, Prop.~\ref{prop:C-eigen} \\
\ref{sec:cbanach}  & Banach type theorem &  Prop.~\ref{degree2}, Theorem~\ref{thm:Banach2} \\ \hline
\ref{sec:gform}    & General conjugate form {and} symmetric tensor  &  Def.~\ref{def:gform}, Lemma~\ref{thm:tensorG}  \\
\ref{sec:gtensor}  & Real-valued general conjugate form {and} conjugate super-symmetric tensor  & Cor.~\ref{thm:condition}, Def.~\ref{thm:css}, Prop.~\ref{thm:gform-rv} \\
\ref{sec:eigendef} & $G$-eigenvalue and $G$-eigenvector & Def.~\ref{thm:Geigen}, Prop.~\ref{prop:G-eigen}\\
\ref{sec:gbanach}  & Banach type theorem  & Theorem~\ref{thm:equal-con} \\ \hline
\end{tabular}
\end{center}
\caption{Summary of the symmetric conjugate form and the general conjugate form}
\label{table}
\end{table}

An important aspect of polynomials is the theory of nonnegativity. Most existing results only apply for polynomials in real variables, for the reason that such polynomials are real-valued. Since we have the full characterization of real-valued conjugate complex polynomials introduced in this paper, the question about their nonnegativity naturally arises, in particular, the relationship between nonnegativity and SOS. In the real domain, this problem was completely solved by Hilbert~\cite{H88} in 1888. However, relationship between nonnegative complex polynomials and SOS has not been established explicitly in the literature as far as we know. This would be one of the future research using the notion of conjugate polynomials. Moreover, the new notions of symmetric complex tensors and the eigenvalues/eigenvectors would hopefully attract future modelling opportunities, and the newly developed properties, in particular the extension of Banach's result would be helpful in solution methods for complex polynomial optimization.

\section*{Acknowledgements}

The authors would like to thank the anonymous referees and the associated editor for their insightful comments, which helped to significantly improve this paper from its original version.
We are also grateful to Augusto Aubry for discussions on radar waveform design with similarity constraint.


\begin{thebibliography}{99}

\bibitem{AK09} T. Aittomaki and V. Koivunen, {\em Beampattern Optimization by Minimization of Quartic Polynomial}, Proceedings of 2009 IEEE/SP 15th Workshop on Statistical Signal Processing, 437--440, 2009.



\bibitem{ADFW13} A. Aubry, A. De Maio, A. Farina, and M. Wicks, {\em Knowledge-Aided (Potentially Cognitive) Transmit Signal and Receive Filter Design in Signal-Dependent Clutter}, IEEE Transactiopns on Aerospace and Electronic Systems, 49, 93--117, 2013.

\bibitem{ADJZ12} A. Aubry, A. De Maio, B. Jiang, and S. Zhang, {\em Ambiguity Function Shaping for Cognitive Radar via Complex Quartic Optimization}, IEEE Trans. Signal Process., 61, 5603--5619, 2013.

\bibitem{B38} S. Banach, {\em \"{U}ber homogene Polynome in $(L^2)$}, Studia Math., 7, 36--44, 1938.

\bibitem{B83} D. H. Brandwood, {\em A Complex Gradient Operator and Its Application in Adaptive Array Theory}, Communications, Radar and Signal Processing, 130, 11--16, 1983.


\bibitem{CC03} E. Carlini and J. V. Chipalkatti, {\em On Waring's Problem for Several Algebraic Forms}, Comment. Math. Helv., 78, 494--517, 2003.

\bibitem{CS13} D. Cartwright and B. Sturmfels, {\em The Number of Eigenvalues of a Tensor}, Linear Algebra Appl., 438, 942--952, 2013.

\bibitem{CHLZ12} B. Chen, S. He, Z. Li, and S. Zhang, {\em Maximum Block Improvement and Polynomial Optimization}, SIAM J. Optim., 22, 87--107, 2012.

\bibitem{CHLZ14} B. Chen, S. He, Z. Li, and S. Zhang, {\em On New Classes of Nonnegative Symmetric Tensors}, Technical Report, 2014.

\bibitem{CGLM08} P. Comon, G. Golub, L.-H. Lim, and B. Mourrain, {\em Symmetric Tensors and Symmetric Tensor Rank}, SIAM J. Matrix Anal. Appl., 30, 1254--1279, 2008.

\bibitem{DDV00} { L. De Lathauwer, B. De Moor, and J. Vandewalle, {\em A Multilinear Singular Value Decomposition}, SIAM J. Matrix Anal. Appl., 21, 1253--1278, 2000.}


\bibitem{H88} D. Hilbert, {\em \"{U}ber die Darstellung Definiter Formen als Summe von Formenquadraten}, Math. Ann., 32, 342--350, 1888.


\bibitem{HLZ10} S. He, Z. Li, and S. Zhang, {\em Approximation Algorithms for Homogeneous Polynomial Optimization with Quadratic Constraints}, Math. Program., Series B, 125, 353--383, 2010.


\bibitem{HLZ13} S. He, Z. Li, and S. Zhang, {\em Approximation Algorithms for Discrete Polynomial Optimization}, J. Oper. Res. Soc. China, 1, 3--36, 2013.

\bibitem{HS10} J. J. Hilling and A. Sudbery, {\em The Geometric Measure of Multipartite Entanglement and the Singular Values of a Hypermatrix}, J. Math. Phys., 51, 072102, 2010.



\bibitem{JLZ14} B. Jiang, Z. Li, and S. Zhang, {\em Approximation Methods for Complex Polynomial Optimization}, {Comput. Optim. Appl.}, 59, 219--248, 2014.

\bibitem{JLZ15} B. Jiang, Z. Li, and S. Zhang, {\em On Cones of Nonnegataive Quartic Forms}, Found. Comput. Math., {Published Online, DOI: 10.1007/s10208-015-9286-4}.

\bibitem{JMZ15} { B. Jiang, S. Ma, and S. Zhang, {\em Tensor Principal Component Analysis via Convex Optimization}, Math. Program., Series A, 150, 423--457, 2015.}

\bibitem{JMZ14} B. Jiang, S. Ma, and S. Zhang, {\em  Alternating Direction Method of Multipliers for Real and Complex Polynomial Optimization Models}, Optimization, 63, 883--898, 2014.

\bibitem{KR02} { E. Kofidis and P. A. Regalia, {\em On the Best Rank-$1$ Approximiation of Hihger-Order Supersymmetric Tensors}, SIAM J. Matrix Anal. Appl., 23, 863--884, 2002.}

\bibitem{KoldaMayo11}{T. G. Kolda and J. R. Mayo, {\em Shifted Power Method for Computing Tensor Eigenpairs}. SIAM J. Matrix Anal. Appl., 32, 1095--1124, 2011.}
%



\bibitem{LHZ12} Z. Li, S. He, and S. Zhang, {\em Approximation Methods for Polynomial Optimization: Models, Algorithms, and Applications}, SpringerBriefs in Optimization, Springer, New York, 2012.

\bibitem{L05} L.-H. Lim, {\em Singular Values and Eigenvalues of Tensors: A Variational Approach}, Proceedings of the IEEE International Workshop on Computational Advances in Multi-Sensor Adaptive Processing, 1, 129--132, 2005.

\bibitem{LMSYZ10} Z.-Q. Luo, W.-K. Ma, A. M.-C. So, Y. Ye, S. Zhang, {\em Semidefinite Relaxation of Quadratic Optimization Problems}, IEEE Signal Processing Magazine, 27, 20--34, 2010.

\bibitem{LT92} Z.-Q. Luo and P. Tseng, {\em On the Convergence of the Coordinate Descent Method for Convex Differentiable Minimization}, J. Optim. Theory Appl., 72, 7--35, 1992.




\bibitem{NQB14} G. Ni, L. Qi, and M. Bai, {\em Geometric Measure of Entanglement and U-Eigenvalues of Tensors}, SIAM J. Matrix Anal. Appl., 35, 73--81, 2014.

\bibitem{NQWW07} G. Ni, L. Qi, F. Wang, and Y. Wang, {\em The Degree of the E-Characteristic Polynomial of an Even Order Tensor}, J. Math. Anal. Appl., 329, 1218--1229, 2007.

\bibitem{NieWang14} {J. Nie and L. Wang, {\em Semidefinite Relaxations for Best Rank-1 Tensor Approximations}, SIAM J. Matrix Anal. Appl., 35, 1155--1179, 2014.}


\bibitem{Q05} L. Qi, {\em Eigenvalues of a Real Supersymmetric Tensor}, J. Symbolic Comput., 40, 1302--1324, 2005.

\bibitem{Q07} L. Qi, {\em Eigenvalues and Invariants of Tensors}, J. Math. Anal. Appl., 325, 1363--1377, 2007.

\bibitem{Q12} L. Qi, {\em The Spectral Theory of Tensors (Rough Version)}, Technical Report, arXiv:1201.3424, 2012.

\bibitem{R91} R. Remmert, {\em Theory of Complex Functions}, Graduate in Texts Mathematics, Springer, New York, 1991.



\bibitem{SZY07} A. M.-C. So, J. Zhang, and Y. Ye, {\em On Approximating Complex Quadratic Optimization Problems via Semidefinite Programming Relaxations}, Math. Program., Series B, 110, 93--110, 2007.

\bibitem{SBD12} L. Sorber, M. Van Barel, and L. De Lathauwer, {\em Unconstrained Optimization of Real Functions in Complex Variable}, SIAM J. Optim., 22, 879--898, 2012.

\bibitem{SBD13} L. Sorber, M. Van Barel, and L. De Lathauwer, {\em Optimization-Based Algorithms for Tensor Decompositions: Canonical Polyadic Decomposition, Decomposition in Rank-$(L_r,L_r,1)$ Terms and a New Generalization}, SIAM J. Optim., 23, 695--720, 2013.

\bibitem{SBD13b} L. Sorber, M. Van Barel, and L. De Lathauwer, {\em Complex Optimization Toolbox v1.0}, \texttt{http://www.esat.kuleuven.be/sista/cot}, 2013.

\bibitem{TO98} O. Toker and H. Ozbay, {\em On the Complexity of Purely Complex $\mu$ Computation and Related Problems in Multidimensional Systems}, IEEE Trans. Automat. Control, 43, 409--414, 1998.

\bibitem{WG03} T. C. Wei and P. M. Goldbart, {\em Geometric Measure of Entanglement and Applications to Bipartite and Multipartite Quantum States}, Phys. Rev. A, 68, 042307, 2003.

\bibitem{ZH06} S. Zhang and Y. Huang, {\em Complex Quadratic Optimization and Semidefinite Programming}, SIAM J. Optim., 16, 871--890, 2006.

\bibitem{ZLQ12} X. Zhang, C. Ling, and L. Qi, {\em The Best Rank-$1$ Approximation of a Symmetric Tensor and Related Spherical Optimization Problems}, SIAM J. Matrix Anal. Appl., 33, 806--821, 2012.

\bibitem{ZQ12} X. Zhang and L. Qi, {\em The Quantum Eigenvalue Problem and Z-Eigenvalues of Tensors}, Technical Report, arXiv:1205.1342, 2012.


\end{thebibliography}
\end{document}